\documentclass[11pt, english]{amsart}
\usepackage{amssymb,amsthm,amsmath,amstext, color}
\usepackage{mathrsfs}
\usepackage{bm}
\usepackage[utf8]{inputenc}
\usepackage{xcolor}
\usepackage%[disable]
{todonotes}
\usepackage{colortbl}

\usepackage{mathrsfs,upgreek}
\usepackage{appendix}
\usepackage{hyperref}

\usepackage[only,mapsfrom,heavycircles,lightning]{stmaryrd}
\usepackage{calc}
\usepackage{paralist}

\makeatletter
\ifnum\@ptsize=0  \fi \ifnum\@ptsize=2  \fi 

\newenvironment{mycpctenum}
{\setlength{\leftmargini}{0mm}
\begin{compactenum}}
{\end{compactenum}}

\makeatletter
\newcommand{\arkfamily}{\fontencoding{U}\fontfamily{ark}\selectfont}
\newcommand{\ark@sym}[1]{{\arkfamily\symbol{#1}}}
\newcommand{\leftthumbsup}{\ark@sym{'125}}
\newcommand{\smallpencil}{\ark@sym{'120}}

\newcounter{genus}
\newcounter{gagid}
\renewcommand{\thegagid}{g\arabic{genus}.\arabic{gagid}}
\@addtoreset{gagid}{section}
\def\gag{\refstepcounter{gagid}
\null\hfill\smallpencil\thegagid}
\makeatother

\begin{document}

\title{On the automorphisms of Mukai varieties}

\author{Thomas Dedieu}
%\address{Institut de Math{\'e}matiques de Toulouse ; UMR 5219 \\ % UPS, F-31062 Toulouse Cedex 9, France}
%\email{thomas.dedieu@math.univ-toulouse.fr}
\author{Laurent Manivel, \linebreak\MakeLowercase{with an appendix by}  Yuri Prokhorov}
%\address{Institut de Math{\'e}matiques de Toulouse ; UMR 5219 \\ % UPS, F-31062 Toulouse Cedex 9, France}
%\email{manivel@math.cnrs.fr}
%\author{\MakeLowercase{with an appendix by} by Yuri Prokhorov}
\date{\today}

\maketitle

\begin{abstract}
Mukai varieties are Fano varieties of Picard number one and coindex three. 
In genus seven to ten they are linear sections of some special homogeneous varieties.
We describe the generic automorphism groups of these varieties. When they are expected
to be trivial for dimensional reasons, we show they are indeed trivial,
up to three interesting and unexpected exceptions in genera $7, 8,9$, and codimension 
$4, 3, 2$ respectively. We conclude in particular that a generic prime Fano threefold
of genus $g$ has no automorphisms for $7\le g\le 10$. In the Appendix by Y. Prokhorov, 
the latter statement is extended to $g=12$.
\end{abstract}

\theoremstyle{plain}
\newtheorem{theorem}{Theorem}
\newtheorem*{conjecture*}{Conjecture}
\newtheorem{prop}[theorem]{Proposition}
\newtheorem{lemma}[theorem]{Lemma}
\newtheorem{coro}[theorem]{Corollary}
\newtheorem{remark}[theorem]{Remark}

\newtheorem{mtheorem}[subsection]{Theorem}
\newtheorem{claim}[subsubsection]{Claim}
\theoremstyle{definition}
\newtheorem{mremark}[subsection]{Remark}
\makeatletter
%\@addtoreset{equation}{subsection}
\makeatother

\let\mathbb\mathbf
\let\subset\subseteq
\let\ge\geqslant
\let\geq\geqslant
\let\le\leqslant
\let\leq\leqslant

\def\AA{{\mathbb{A}}}
\def\CC{{\mathbb{C}}}
\def\RR{{\mathbb{R}}}
\def\OO{{\mathbb{O}}}
\def\HH{{\mathbb{H}}}
\def\PP{{\mathbb{P}}}
\def\QQ{{\mathbb{Q}}}\def\ZZ{{\mathbb{Z}}}
\def\SS{{\mathbb{S}}}
\def\cO{{\mathcal{O}}}
\def\O{{\mathcal{O}}}
\def\cE{{\mathcal{E}}}
\def\cF{{\mathcal{F}}}
\def\cG{{\mathcal{G}}}
\def\cL{{\mathcal{L}}}
\def\cU{{\mathcal{U}}}
\def\cV{{\mathcal{V}}}
\def\ra{{\rightarrow}}
\def\lra{{\longrightarrow}}
\def\ft{{\mathfrak t}}
\def\fg{{\mathfrak g}}
\def\fso{\mathfrak{so}}
\def\fsp{\mathfrak{sp}}
\def\fsl{\mathfrak{sl}}
\def\tsum{{\textstyle\sum}}

\def\Ad{\mathrm{Ad}}
\def\Id{\mathrm{Id}}
\newcommand{\Sn}[1][n]{\mathfrak{S}_{#1}}

\def\Gr{\mathrm{Gr}}
\def\SO{\mathrm{SO}}
\def\GL{\mathrm{GL}}
\def\PGL{\mathrm{PGL}}
\def\Sp{\mathrm{Sp}}
\def\PSp{\mathrm{PSp}}
\def\G{\mathrm{G}}
\def\LG{\mathrm{LG}}
\def\IG{\mathrm{IG}}
\def\Pf{\operatorname{Pf}}
\def\PSL{\mathrm{PSL}}
\def\SL{\mathrm{SL}}
\def\Spin{\mathrm{Spin}}
\def\Gm{\mathbb{G}_m}
\def\Ga{\mathbb{G}_a}

\def\Aut{\mathrm{Aut}}
\def\Hom{\mathrm{Hom}}
\def\Stab{\mathrm{Stab}}
\def\codim{\mathrm{codim}}

\let\Iff\iff
\def\iff{\Leftrightarrow}
\def\epsilon{\varepsilon}
\newcommand{\acts}{\circlearrowleft}
\newcommand{\pprime}{%
  \mbox{\raisebox{.05em}{\smaller[5](}}%
  \prime%
  \mbox{\raisebox{.05em}{\smaller[5])}}%
  }

\tableofcontents

\section{Introduction} 
The classification of Fano threefolds by Fano, Iskovskih, and finally Mori and Mukai is a milestone in 
the history of complex algebraic geometry \cite{ip}. From this classification, and the subsequent work
of Mukai, a prime Fano threefold of genus $g=7,8,9,10$ turns out to be a linear section of a complete 
$G$-homogeneous variety 
$M_g\subset \PP (V_{g})$, for some simple algebraic group $G$ of which $V_{g}$ is an irreducible 
representation. These groups, representations and homogeneous varieties are recalled in the 
table below (the notation $k_g$ is introduced above
Theorem~\ref{stab-Fano}):\medskip 

\begin{center}
\begin{tabular}{c|cccccc|c}
  $g$ & $G$ & $\dim(G)$ & $V_g$ & $\dim(V_g)$ & $M_g$
  & $\mathrm{dim}(M_g)$ & $k_g$ \\
  \hline
  $7$ & $\Spin_{10}$ & $45$ & $\Delta_+$ & $16$ & $\SS_{10}$ & $10$
  & $4$ \\
$8$ & $\SL_6$ & $35$ & $\wedge^2\CC^6$ & $15$ & $\G(2,6)$ & $8$ & $3$ \\
$9$ & $\Sp_6$ & $21$ & $\wedge^{\langle 3\rangle}\CC^6$ & $14$ &
                                                                 $\LG(3,6)$
  & $6$ & $2$ \\
$10$ & $G_2$ & $14$ & $\fg_2$ & $14$ &  $G_2/P_2$ & $5$ & $2$
\end{tabular}
\end{center}

\medskip\noindent
Here $\Delta_+$ denotes one of the half-spin representations of
$\Spin_{10}$. On the other hand $\wedge^{\langle 3\rangle}\CC^6
\subset \wedge^3\CC^6$ is defined by the condition that the contraction by a two-form that is invariant under $\Sp_6$,
vanishes. More details on these representations and varieties will be provided in the relevant sections. 

\medskip
The following numerical relations hold:
$$ \mathrm{codim} (M_g)=g-2, \qquad \mathrm{index} (M_g)=\mathrm{dim}(M_g)-2.$$
Fano varieties $X$ with coindex $\mathrm{dim}(X)+1-\mathrm{index}(X)=3$ are called 
{\it Mukai varieties}. They were classified by Mukai \cite{mukai-coindex3}, modulo a conjecture 
on the existence of smooth  canonical divisors which was later proved
by Mella \cite{mella},
see also \cite{CLM98} for a different approach based on the theory of
\emph{extensions}.
Note that the coindex is preserved by taking hyperplane sections, so it is enough to classify 
maximal Mukai varieties, those that are not hyperplane sections of any smooth variety. The homogeneous
varieties $M_g$ are precisely the maximal Mukai varieties of degree $2g-2$, for $7\le g\le 10$. The Mukai
varieties of those degrees are thus the smooth linear sections of the minimally embedded 
$M_g\subset \PP (V_g)$. These are the varieties we study in this paper. 

Of course $M_g$ has a big automorphism group. Its hyperplane sections also admit non trivial automorphisms. In genus $g=7,8,9$ there is in fact 
a unique smooth hyperplane section, up to the action of $G$: the representations $V_g$ (or rather their duals) are prehomogeneous, and the claim readily follows from the easy classification of the 
$G$-orbits. Moreover the smooth hyperplane section is acted on by the generic stabilizer of $\PP(V_g^*)$,
which was computed in \cite{sk}.
In genus $g=8$ one obtains the isotropic 
Grassmannian $\IG(2,6)$, which is homogeneous under the action of the symplectic group $\Sp_6$. In 
genus $g=7$, the automorphism group of the hyperplane section is not reductive and its action
is only prehomogeneous.  In genus $g=9$, the automorphism group is reductive, but too small to act
on the hyperplane section with an open orbit. In genus $10$, a generic element of $V_{10}^*=\fg_2$ is 
a regular semisimple element in the Lie algebra, which is stabilized by a maximal torus in $G_2$; 
up to the action of $G_2$ there is therefore a one dimensional family of hyperplane sections of $G_2/P_2$, 
whose connected automorphism group is a two dimensional torus (see \cite{pz1, pz2} for a recent study).

There are a few other small codimensional sections of $M_g$ with non
trivial automorphisms, the existence of which can be deduced from the
fact that the action of $\GL_k\times G$ on $\CC^k\otimes V_g$
is prehomogeneous, for some small integers $k>1$, with non trivial generic stabilizer. This happens 
for $g=7$, $k=2,3$, and for $g=8$, $k=2$ (see \cite{bfm} for the connection with exceptional Lie groups). 
The case $g=7, k=2$ was studied in detail in \cite{kuz-s10}. 

Let us compile those $k$-codimensional linear sections $X\subset M_g$ with non trivial automorphisms 
in the table below, with their generic connected automorphism groups  (possibly up to some finite group). 
\medskip 

\begin{center}
\begin{tabular}{cccl}
$g$ & $k$ & $\Aut^0(X)$ &  \cite{sk} \\
$7$ & $1$   &  $(\Gm\times \Spin_7)\times \Ga^8$ & Proposition 31 p.121 \\
 & $2$  &  $G_2\times \SL_2$ & Proposition 32 p.124 \\
 & $3$  &  $\SL_2^2$ & Proposition 33 p.126  \\
$8$ & $1$ &  $\Sp_6$ & \\
 & $2$ &  $\SL_2^3$ & Proposition 12 p.94\\
$9$ & $1$ & $\SL_3$ & Proposition 22 p.108\\
$10$ & $1$ & $\Gm^2$ & \cite{pz1}
\end{tabular}
\end{center}

\medskip\noindent We are thus led to let $k_g=4,3,2,2$ for $g=7,8,9,10$, respectively. 
The first main result of this paper can then be stated as follows. 

\begin{theorem}\label{stab-Fano}
For $g=7,8,9,10$, a generic linear section $X$ of $M_g$, of dimension at least three, and of 
codimension $k\ge k_g$, has only trivial automorphisms, except for $k=k_g$ and $g=7, 8, 9$. 
\end{theorem}

Quite surprinsingly, we could not find this result in the literature for the well-studied case
of Fano threefolds. For any smooth prime Fano threefold of genus $g<12$,  the automorphism group is known
to be finite by \cite{kps}. It is known to be trivial for a general prime Fano threefold of genus $6$ 
\cite[Proposition 3.21]{dk}, but the corresponding statement in higher genus seems new. Recall that 
the genus of a prime Fano threefold cannot be greater than twelve, nor be equal to eleven. The case 
where $g=12$, which requires a different approach, is treated in the Appendix by Yuri Prokhorov. 

\begin{coro} 
The automorphism group of a general prime Fano threefold of genus $g\ge 7$ is trivial. 
\end{coro}

%\medskip
Our strategy to prove Theorem \ref{stab-Fano}
will be to reduce this statement to the following one. Denote by $\bar{G}$ the 
image of $G$ in $\PGL(V_g)$.

\begin{theorem}\label{stab-sous-espace}
For $g=7,8,9,10$, let $L\subset V_g$ be a generic linear subspace
such that
\begin{equation}
\label{eq:c-dim.stab}
\tag{$\star$}
  \min\bigl(\codim_{V_g}(L),\dim(L)\bigr) \geq k_g.
\end{equation}
Then the stabilizer of $L$ in $\bar{G}$ is trivial, except
if equality holds in \eqref{eq:c-dim.stab} and $g=7, 8, 9$. 
\end{theorem}

In order to deduce Theorem \ref{stab-Fano} from Theorem \ref{stab-sous-espace}, we need to prove that
any automorphism of $X=M_g\cap\PP(L)$ must be induced by an element of $G$ stabilizing $L$ (at least 
when the latter is generic). This kind of statement is at the heart of Mukai's approach to prime Fano
threefolds and $K3$ surfaces of small genus. More precisely, each Mukai variety of genus $g=7,8,9,10$ 
admits a unique special vector bundle $E_g$ of rank $r_g=5,2,3,2$, which defines its embedding into $M_g$, 
itself naturally embedded in a Grassmannian of rank $r_g$ subspaces.
The case $n=3$ of the statement below has been proved by
Mukai \cite[Theorem 0.9]{mukai-K3Fano} (for genus $10$, 
see  Proposition 5.1 in op.\ cit.\ and the discussion that follows).
The case $n=2$ is also claimed in \cite[Theorem 0.2]{mukai-K3Fano}, 
but the proof seems to apply 
only under a stability assumption (see \cite[Section 2, (2.2)]{mukai-K3Fano}). 
The latter stability assumption holds as soon as the Picard group is
generated  
by the hyperplane line bundle, a condition which is always fulfilled
in dimension $n \geq 3$ by the Lefschetz theorem,
but holds under a very generality assumption in
dimension $n=2$ by a suitable version of the Noether--Lefschetz theorem.

\begin{prop}\label{p:mukai}
Let $X=M_g\cap \PP(L)$ and $X'=M_g\cap \PP (L')$ be smooth linear sections of 
$M_g$, of the same dimension $n\geq 3$, and suppose that $\varphi: X\rightarrow X'$ is an isomorphism. Then there exists $g\in G$ such
that $L'=g(L)$ and $\varphi = g^*$.
\end{prop}

\begin{proof}
Mukai first proves the corresponding statement for general
$K3$ surfaces, and then 
deduces it for Fano threefolds. 
It suffices to check that Mukai's argument for the latter point
extends to the case when $X$ and $X'$ have dimension  
larger than three. The key point 
is to make use of the bundle $E_g$, and its restrictions $F$ and $F'$ to $X$ and $X'$, respectively. 
These bundles must be stable, because they are stable on a general
surface linear section, thanks to the openness of
the stability condition. Moreover the 
restriction of $F$ to a general $K3$ section $S\subset X$, and the restriction of $F'$ to $\varphi(S)$ 
must be isomorphic (there is a unique such bundle on a general $K3$ surface of genus $g$, see \cite[Section 2, (2.2), Step I]{mukai-K3Fano}). 
This isomorphism then 
lifts to an isomorphism between $F$ and $\varphi^* (F')$
(this may be proved by induction on $n$, using almost verbatim the
argument given in the proof of \cite[Proposition 5.1]{mukai-K3Fano}). 
Eventually the discussion of \cite[Section 2]{mukai-K3Fano},
almost without change, shows that this isomorphism has to be induced by the linear action of some element 
of $G$.  Indeed the key cohomological arguments in \cite[Section 2, (2.2), Step II and Step III]{mukai-K3Fano}
are deduced from Bott's theorem applied to some Koszul complexes, and there are all the more vanishing to check that
the codimension is larger. So this last part is actually less demanding 
for $n> 3$ than it is for $n=3$, in the sense that all the required cohomological vanishing that are required have already been checked by Mukai when he discussed the latter case. 
\end{proof}

The case $n=1$ of Proposition~\ref{p:mukai} holds as well
\cite{mukai-symm0}, but requires a
different approach.
A different proof is also given in \cite[\S 4]{cd}, as well as variations
on Proposition~\ref{p:mukai}.
It is on the other hand well-known that Theorem~\ref{stab-Fano} also
holds for curves and surfaces linear sections of $M_g$.

\medskip There are three exceptional cases in the previous statements, in genus $7, 8, 9$, for which non 
trivial finite groups of symmetries show up unexpectedly. This is the second main result of this paper. 

\begin{theorem}\label{symetries}
For $g=7, 8, 9$, let $L\subset V_g$ be a generic linear subspace of codimension $k_g$.
Denote by $\Aut_G(L)$ the image in $\PGL(L)$ of the stabilizer $\Stab(L)\subset G$.  
Let $X=M_g\cap\PP(L)$ be the corresponding Mukai variety. 
Then 
$$\Aut(X)= \Aut_G(L)=\begin{cases} (\ZZ/2\ZZ)^2 & \mathit{for\;} g=7, \\ 
\CC^*\rtimes  ((\ZZ/3\ZZ)^2\rtimes \ZZ/2\ZZ) & \mathit{for\;} g=8, \\ 
 (\ZZ/2\ZZ)^4 & \mathit{for\;} g=9.\end{cases}$$ 
\end{theorem}

Moreover in all cases of the above statement we are able to identify
the fixed locus of the various automorphisms. Actually the genus $8$ case has already  
been discussed by Piontkowski and Van de Ven, who obtained by direct computations a less 
precise result \cite[Theorem 4.6]{pv}. 

Proving these two statements will require a careful analysis, the principles of which are explained in the 
next section. Geometrically, we conclude that 
the generic codimension four sections of $M_7$, and the generic codimension two sections 
of $M_9$ are stabilized by finitely many non trivial  involutions that we will describe 
explicitely. The case of $g=8$ is more specific and less unexpected since by duality, there is 
a plane cubic associated to a codimension three section of $M_8$, on which the action of $\Aut_G(L)$
can be read off. 

In a subsequent paper, we will interpret the previous Theorem in terms of 
$\theta$-representations, and show that our small groups of automorphisms can be seen as 
traces of complex reflection groups defined as generalized Weyl groups of some graded Lie algebras. 

\medskip\noindent {\it Acknowledgements.} We thank 
Yuri Prokhorov for his comments on the automorphisms of prime Fano
threefolds, and his permission to include his Appendix on the genus
twelve case. We also warmly thank the anonymous referees for their
careful reading, and their suggestions which allowed in particular to
drastically simplify the proof of the crucial Lemma~\ref{injectivity}.
We are still thankful to Christian Krattenthaler for his kind help
with some determinants that appeared in the proofs of
Propositions~\ref{Jordan-tensor} 
and \ref{Jordan-wedge} in the first version of this article, even
though the arguments have now been modified following a 
suggestion of a referee.

We acknowledge support from the ANR project FanoHK, grant
ANR-20-CE40-0023.

\section{Jordan decomposition miscellany}

\subsection{Stable subspaces}

Let $S\subset \bar{G}$ be the stabilizer of a generic subspace $L\subset V_g$. In the relevant
range of dimensions, we will prove that $S$ is trivial by proving that it contains no non trivial
semisimple or unipotent element. Indeed, if $g\in G$ stabilizes $L$, one can use its Jordan decomposition
$g=g_sg_n$ in $G$, and observe that since $g_s$ and $g_n$ are polynomials in $g$ (once considered
as elements of $\Hom(V_g)$), they must also stabilize $L$. 

This reduction will allow us to treat separately unipotent and semisimple elements. We will stratify
the set of those elements and control for each stratum the dimension of the stable subspaces.
Then a simple dimension count will imply that the generic $L$ has no stabilizer. This dimension count 
will be based on a straightforward bound for the variety of $m$-dimensional spaces stabilized by a 
unipotent or semisimple endomorphism, in terms of its Jordan type. 

\begin{prop}\label{dim-stable}
For $g\in \GL(V)$, let $G_m(g)\subset \G(m,V)$ denote the variety of $m$-dimensional subspaces 
which are stabilized by $g$.
{\setlength\leftmargini{\widthof{\it (2)\ }}
\begin{compactenum} 
\item \raisebox{0mm}[1em]
If $g$ is semisimple with eigenvalues of multiplicities $e_1, \ldots, e_p$, the dimension 
of $G_m(g)$ is bounded by the maximum of the 
$$f_1(e_1-f_1)+\cdots + f_p(e_p-f_p)$$ for $0\le f_i\le e_i$ and
$f_1+\ldots +f_p=m$.
\smallskip
\item If $g$ is unipotent with $b_k$ Jordan blocks of size $k$ for
  $1\le k\le q$,
  we let $\beta_p = b_q+\cdots+b_p$ for all $p=1,\ldots,q$,
and then the dimension 
of $G_m(g)$ is bounded by the maximum of the $$\gamma_1(\beta_1-\gamma_1)+\cdots +\gamma_q(\beta_q-\gamma_q),$$
taken over the sequences $\gamma_1\ge\cdots \ge \gamma_q\ge 0$ such that $\gamma_1+\cdots + \gamma_q=m $
and $\gamma_p\le \beta_p$ for any $p\le q$. 
\end{compactenum}
}
\end{prop}

\proof If $g$ is semisimple and its eigenspace decomposition is $V=E_1\oplus \cdots \oplus E_p$, 
simply observe that a subspace $L$ stabilized by $g$ must be of the form $L=F_1\oplus \cdots \oplus F_p$
for $F_i\subset E_i$. 

If $g=id+X$ is unipotent, we construct a stable subspace $L$ such that the restriction $Y$ of $X$ to $L$ has $c_i$ blocks of size $i$, for $1\le i\le q$, as follows. We first choose a subspace $L_q$, of dimension $c_q$, transverse to $\ker(X^{q-1})$; this is possible for $c_q\le b_q$, and then it is an open, non empty condition.
Inductively, for any $1\le p<q$, we then choose a subspace $L_p$, of dimension $c_q+\cdots +c_{p}$, 
such that $XL_{p+1}\subset L_p\subset \ker(X^{p})$, transverse to $\ker(X^{p-1})$; the latter condition 
can be realized only when $c_q+\cdots +c_p\le b_q+\cdots +b_p$, and then it is an open, non empty condition. 
Finally we let $L=L_1+\cdots +L_q$. By the transversality conditions we imposed, this is a direct sum and 
therefore, the dimension of $L$ is $m=c_1+2c_2+\cdots +qc_q$. Moreover, by construction $L$ is stable with
the prescribed Jordan type, and every such $L$ can be obtained like that. In terms of the dimensions $x_p$ of $\ker(X^p)$, given by $x_p=\sum_{k=1}^q\min(k,p)b_k$, we can express the number of parameters for 
 $(L_1,\ldots ,L_q)$ as 
 $$\sum_{p=1}^q c_p(x_p-c_q-\cdots -c_p).$$
 Now, $L$ being given, we can choose $(L_1,\ldots ,L_q)$ inside $L$ subject to the same conditions as above;
 the number of parameters for $(L_1,\ldots ,L_q)$ is then given by the same formula, but with $x_p$ replaced
 by the dimension of $\ker(Y^p)$; which is $y_p=\sum_{k=1}^q\min(k,p)c_k$. Finally, the number of parameters for $L$ itself is the difference between these two numbers. Letting $\beta_p=b_q+\cdots +b_p$ and 
 $\gamma_p=c_q+\cdots +c_p$ we get the result announced. 
\qed 

\subsection{Jordan types of tensor products} 
In the sequel we will meet representations $W$ defined as tensor products $U\otimes V$, and we will 
need to control the Jordan type of an endomorphism $Z\in \Hom(W)$ defined from two endomorphisms 
$X\in \Hom(U)$ and $Y\in\Hom(V)$, as $Z=X\otimes \Id_V+\Id_U\otimes Y$. Let $u$, $v$ denote the dimensions
of $U$ and $V$. Recall that a nilpotent endomorphism is called {\it regular} if it has only one 
Jordan block. 

\begin{prop}\label{Jordan-tensor}
Suppose that $X\in \Hom(U)$ and $Y\in \Hom(V)$ are regular nilpotent, and that $u\ge v$.  Then $Z$ has $v$ 
Jordan blocks, of sizes $u-v+1, u-v+3,\ldots , u+v-3,u+v-1$.  
\end{prop}

\proof Let us include $X$ and $Y$ into $\fsl_2$-triples, meaning that we find $H,X'\in\Hom(U)$ such that 
$$ [H,X]=2X, \quad [H,X']=-2X', \quad [X,X']=H,$$
and similarly for $Y$ (see \cite[Corollary 3.2.7]{cm}). In particular $(H,X,X')$ generate a subalgebra of $\Hom(U)$ which 
is isomorphic to $\fsl_2$, 
and the fact that $X$ is regular can be translated into the fact that $U$ is an irreducible module over this 
copy of $\fsl_2$; and similarly for $Y$ and $V$. 
If we denote by $M_k$
the unique irreducible $\fsl_2$-module of dimension $k+1$, the Clebsch-Gordan formula yields 
$$U\otimes V=M_{u-1}\otimes M_{v-1} = M_{u+v-2}\oplus M_{u+v-4}\oplus\cdots \oplus M_{u-v+2}\oplus M_{u-v}.$$
Each factor $M_k$ in this decomposition yields a Jordan block of size $k+1$ for $Z$, hence the claim. 
\qed

\medskip\noindent
We will also need a skew-symmetric version of Proposition \ref{Jordan-tensor}, where we consider the action of
an element $X\in \Hom(U)$ on $\wedge^2U$. We will denote the induced operator by $\wedge^2X$. 

\begin{prop}\label{Jordan-wedge}
Suppose that $X\in \Hom(U)$ is regular nilpotent. 
\begin{itemize}
\item If $u=2v$ is even,  $\wedge^2X$ has $v$ Jordan blocks, of sizes $1,5,\ldots 2u-3$.  
\item If $u=2v+1$ is odd, $\wedge^2X$ has $v$ Jordan blocks, of sizes $3,7,\ldots 2u-3$.  
\end{itemize}
\end{prop}

\proof As for the previous result this follows from the classical formula for the decomposition of 
$\wedge^2M_k$ into irreducible components: 
$$\wedge^2U=\wedge^2M_{u-1} = M_{2u-4}\oplus M_{2u-8}\oplus\cdots $$
Again each factor $M_k$ in this decomposition yields a Jordan block of size $k+1$ for $\wedge^2X$, 
hence the claim. 
\qed

\subsection{General strategy} 
\label{s:anciliary}
We shall proceed to a case by case study of the linear sections of the maximal Mukai varieties
$M_g$, for $g=7,8,9,10$. As we already explained, we will check that a general $L\subset V_g$, 
whose dimension $m$ belongs to the relevant range, cannot be
stabilized by any non-trivial unipotent or 
semisimple element in $\bar{G}$, except in the special
cases listed in Theorem~\ref{symetries}, for which our analysis will
show that there are no non-trivial unipotent elements in the
stabilizer, and provide a short list of possible semisimple elements
stabilizing $L$.
The discussion of these two cases will proceed along the
lines indicated in the following two paragraphs.

For each of the special cases of Theorem~\ref{symetries}, we provide
specific representation theoretic arguments to give a definitive
description of the stabilizer. These shall be introduced in due time.

\subsubsection{Unipotent elements}
\label{s:strategy-nilp}
Equivalently, we will check that a generic $L$ 
of dimension $m$ cannot be preserved by a non trivial nilpotent element in the Lie algebra $\fg$ 
of $G$. For this we will use the fact that $\fg$ has only finitely many nilpotent orbits $\cO$,
for each of which we can provide a representative $X$. Then $X$ acts on $V_g$ as a nilpotent 
operator, with a Jordan decomposition that we will determine; Proposition \ref{Jordan-tensor}
and Proposition \ref{Jordan-wedge} will be extremely useful for that. Using Proposition \ref{dim-stable},
we will then deduce the dimension $d_m(\mathcal{O})$ of the variety of $m$-dimensional subspaces 
of $V_g$ stabilized by $X$. The claim we are aiming for will then follow from the inequalities
\begin{equation}
\tag{$\bigstar$}\label{OKnilp}
\dim\mathcal{O}+d_m(\mathcal{O})<\dim \G(m,V_g) \qquad \forall
\mathcal{O}\ne\{0\}.
\end{equation}
In fact it is sufficient to prove the non-strict inequality in the
above condition: to see this we consider
the projection map $\pi: (L,X) \mapsto L$, 
defined on the incidence variety
$\mathcal I_{\O} \subset \Gr(m,V_g) \times \O$
parametrizing pairs $(L,X)$ such that $X.L=L$; 
our claim follows from the fact that the fibres of $\pi$ always have
dimension at least $1$,
which in turn comes from the observation that if $L$ is stabilized by
some $X \in \O$ then it is also stabilized by all multiples
$\lambda X$, $\lambda \in \CC^*$.

We shall use a Python script, described in more details below, 
to verify \eqref{OKnilp} for all nilpotent orbits $\O$:
for each $\O$, we exhaustively list all possible sequences
$(\gamma_i)$ in the notation of Proposition~\ref{dim-stable}, and thus
compute a bound for $d_m(\O)$.

In practice we proceed as follows:
we give the list of all nilpotent orbits, including their dimensions
and the Jordan decompositions for the actions of their members on the
representation $V_g$, and then we use the Python toolkit contained in
the file
\texttt{stab\_nilp.py} to verify the inequality \eqref{OKnilp} for
each 
of them~: our strategy is to exhaust all the possibilities listed in
Proposition~\ref{dim-stable}~(2) in order to compute the maximum. 
The data in our Python format together with the function
calls for the verification for genus $g$ is contained in the file
\texttt{g**.py}, where \texttt{**} is the value of $g$.
We have found that \eqref{OKnilp} always holds
(with equality in the sole case $g=8$ and $m=3$ or $12$, which is fine
as well as noted above), so that the generic $m$-plane $L$ in $V_g$
has no non-trivial unipotent element in its stabilizer if
$k_g \leq m \leq \dim(V_g)-k_g$.
We provide all the values computed by implementing
Proposition~\ref{dim-stable}~(2) in the output files
\texttt{g**\_output.txt}.
We emphasize however that reading through these output files may not
be the most convenient way of using our python tools, and that it is
arguably wiser to trust python on verifying \eqref{OKnilp} for each
case and eventually letting us know if everything was fine. This may
be done by setting the variable \texttt{synthetic} to \texttt{True} in
the files \texttt{g**.py}.

% In the course of the text, we content
% ourselves with listing the orbits and their descriptions, and leave
% the rest in the anciliary files.

\subsubsection{Semisimple elements}
\label{strategy-smspl}

In order to check that a generic $L$ of dimension $m$ cannot be
preserved by a non trivial semisimple element of $G$, we will make a
similar dimension count. First observe that for $g\le 9$, the
representation $V_g$ has only weights of multiplicity one. This
implies that a generic semisimple element acts on $V_g$ with
multiplicity one eigenvalues; in particular, it stabilizes only
finitely many subspaces of $V_g$. These eigenvalues will be obtained
by including our semisimple element into some maximal torus $T\subset
G$ and making use of the weight decomposition of $V_g$ with respect to
this torus.

Positive dimensional families of stable subspaces will only occur when
some of the eigenvalues will coincide,
and we shall carefully classify the possible coincidences.
In effect we will consider the stratification of $G$ determined by
these coincidences, where each stratum parametrizes those elements of
$G$ for which a given set of coincidences happen, but no other.

Each type of coincidence amounts to some polynomial equations verified
by the values taken by the roots.
For a given set of coincidences, let $\mathcal{W}$ be the
locally closed subset of values solutions to the polynomial equations
characterizing our given coincidences, but to no other. Then the
corresponding stratum is the disjoint union 
$S = \coprod _{\mathbf w \in {\mathcal W}} \O_{\bf w}$
of the conjugacy classes in $G$ attached to the values $\mathbf w \in
{\mathcal W}$.

For each such stratum $S$
we will use Proposition~\ref{dim-stable} in order to compute
(or at least bound) the dimension $d_m(S)$ of the variety of
$m$-dimensional subspaces of $V_g$ stabilized by an element of $S$.
To show that a generic $m$-dimensional $L$ has trivial stabilizer, it
is sufficient to prove the inequalities
$$
\forall S\ne\{1\}: \quad
\dim (S)+d_m(S)<\dim \Gr(m,V_g).$$
In fact this criterion may be improved as follows.

\begin{lemma}\label{rk:cleverL}
Let $S = \coprod _{\mathbf w \in {\mathcal W}} \O_{\bf w}$ be a
stratum as above.
The stabilizer of a generic $m$-dimensional $L$ does not intersect $S$
as soon as
\begin{equation}\label{OK}
\tag{\leftthumbsup}
  \forall \mathbf w \in \mathcal W:\quad
  \dim (\O_{\mathbf w})+d_m(S)<\dim \Gr(m,V_g).
\end{equation}
\end{lemma}

\begin{proof}
Consider the projection map $\pi: (L,\gamma) \mapsto L$ defined on the
incidence variety $\mathcal I_S \subset \Gr(m,V_g) \times S$
parametrizing pairs $(L,\gamma)$ such that $\gamma.L=L$.
For a pair $(L,\gamma)$ with $\gamma$ semisimple, the condition
$\gamma.L=L$ is equivalent to $L$ being a direct sum of subspaces of 
the eigenspaces of $\gamma$. But then the eigenvalues are irrelevant, so
for each pair $(L,\gamma) \in \mathcal I_S$
there is in fact a whole family of pairs $(L,\gamma_{\bf w}) \in
\mathcal I_S$ obtained by letting the eigenvalues of $\gamma$ move in
$\mathcal W(S)$.
Therefore the generic fibre of $\pi$ has dimension at least
$\dim \mathcal{W}$.
The claim follows since
$\dim (S) = \dim \mathcal{W} + \dim (\O_{\bf w})$
for any $\mathbf w \in \mathcal{W}$ (see Remark~\ref{l:dim-conjcl}
below).
\end{proof}

For each value of $\mathbf{w}$ the dimension of
$\O_{\bf w}$ may be computed by considering the adjoint action of
$G$ on its Lie algebra $\mathfrak{g}$, as follows.

\begin{remark}\label{l:dim-conjcl}
Let $\gamma \in G$. The conjugacy class of $\gamma$ has codimension
$\mathrm{rk}(G)+\delta$ in $G$, where $\delta$ is the number of roots
of $G$ taking the value $1$ on $\gamma$
(in particular, this codimension is constant along the strata as above).
Indeed, the tangent space at $1\in G$ to the stabilizer of
$\gamma$ for the adjoint action is
$\ker(\Ad(\gamma)-\mathrm{id}_{\mathfrak g})$.
\end{remark}

To proceed with this strategy, we have written an elementary piece of
Python code (included in the ancillary files)
to automatize the computation of the maximum in
Proposition~\ref{dim-stable} (which is done by trying all possible
cases), and the verification of the inequalities 
\eqref{OK} of Lemma~\ref{rk:cleverL}.
In practice we also reduce the cases to be checked by using the
following monotonicity property.

\begin{remark}\label{rk:monotone}
If an eigenspace decomposition is obtained from another one
by breaking the eigenspaces into smaller pieces, then the dimension
$d_m$ of the family of stable $m$-dimensional subspaces will be
larger for the decomposition with larger eigenspaces.
\end{remark}

It is also important to take into consideration the action
of the Weyl group on the roots of $G$ in order to reduce the various
cases to be checked.

To structure our analysis, we shall distinguish two kinds of
coincidences among the eigenvalues for the action of $\gamma \in G$ on
$V_g$, namely
(i) \emph{degenerations}, which are the relations gotten when a
root of $G$ takes the value $1$, and
(ii) \emph{collapsings} which are the
other coincidences between the weights of $V_g$. In particular
collapsings have no effect on the dimension of the conjugacy class.

We encode the decomposition of $V_g$ into eigenspaces as a partition of
$n=\dim(V_g)$, which we write as $[\mu_1^{a_1},\ldots,\mu_p^{a_p}]$ if
there are $a_i$ eigenspaces of dimension $\mu_i$ for $i=1,\ldots,p$.
When listing eigenvalues, we indicate the multiplicity between
parentheses. When we write ``\eqref{OK} holds for all
$m$'', we intend that it is so if $k_g \leq m \leq n-k_g$.
To help locate the exceptional cases, i.e., those which may give rise
to a non-trivial stabilizer for the general subspace, we indicate them
with a '\smallpencil' sign together with a label including the genus.

Our Python toolkit for the semi-simple case is in the file
\texttt{stab\_smspl.py}, and the specific data for the genus $g$ is in
the file \texttt{g**.py}. All values computed by the implementation of
Proposition~\ref{dim-stable}~(1) are provided in the output files 
\texttt{g**\_output.txt}.
However we advise again for the setting of the variable
\texttt{synthetic} in the files \texttt{g**.py} to \texttt{True} in
order to let python handle these outputs and only letting us know
those cases for which something noticeable has happened.

At some points in analyzing the
possible collapsings, it is convenient to use in addition
Macaulay2 to perform some elementary but tedious polynomial
manipulations: the relevant files in these cases are
\texttt{g**\_collapse.m2}.

\section{Genus $8$}
\setcounter{genus}{8}

We will start with the genus $8$ case, which is the easiest one since it only involves the projective linear group 
and its familiar action on the Grassmannian $M_8=\G(2,6)$, embedded in $\PP(V_8)=\PP(\wedge^2\CC^6)$ by the 
usual Pl{\"u}cker embedding. As outlined in our general strategy, we will analyse the possibility for a given 
unipotent or semisimple element of $G=\PSL_6$, to stabilize a generic subspace $L$ of $\wedge^2\CC^6$. Once this is done, 
we will conclude that the stabilizer $S_L$ of a generic $L$ of dimension $4$ to $11$ must be trivial. Moreover, 
if $L$ has dimension $3$ or $12$, its stabilizer can only contain  
involutions and order three elements of a very specific type, which will allow us to  determine completely 
the structure of $S_L$ and $\Aut_G(L)$. This will be the conclusion of a lengthy and laborious analysis that the 
reader may easily skip, in case she is ready to trust the authors.

\subsection{Unipotent elements}
There are eleven nilpotent orbits in $\fsl_6$, corresponding to the eleven partitions $\pi=(\pi_1\ge\cdots\ge \pi_m)$
of six. A representative of $\mathcal{O}_\pi$ is obtained by choosing a splitting $\CC^6=U_{1}\oplus
\cdots\oplus U_{m}$, where $U_i$ has dimension $\pi_i$, and letting $X_\pi=X_1+\cdots +X_m$, with $X_i$ a regular nilpotent element in 
$\fsl(U_{i})\subset\fsl_6$. The Jordan type of the action of $X_\pi$ on $\wedge^2\CC^6$ can then 
be obtained by decomposing 
$$\wedge^2\CC^6 = \Big(\bigoplus_{i=1}^m\wedge^2U_i\Big)\oplus \Big(\bigoplus_{j<k}U_j\otimes U_k\Big)$$
and applying Proposition \ref{Jordan-tensor} and Proposition \ref{Jordan-wedge}. The result is the 
following:
$$\begin{array}{lll}
Partition \qquad  & Dimension  \qquad 
& Jordan\; type\\
6  & 30 & 9,5,1\\ 5,1 & 28 & 7,5,3 \\ 4,2 & 26 & 5^2,3,1^2 \\ 4,1^2 & 24 & 5,4^2,1^2 \\ 
3^2 & 24 & 5,3^3,1 \\ 3,2,1 & 22 & 4,3^2,2^2,1 \\ 3,1^3 & 18 & 3^4,1^3 \\
2^3 & 18 & 3^3,1^6 \\ 2^2,1^2 & 16& 3,2^4,1^4 \\ 2,1^4 & 10 & 2^4,1^7 \\ 1^6  & 0& 1^{15}
\end{array}$$

Arguing as indicated in \S \ref{s:strategy-nilp}, we conclude that for
$m=3,\ldots,12$, the general $m$-dimensional linear subspace $L \subset
\wedge^2 \CC^6$ has no nilpotent element in its stabilizer.

\medskip\noindent {\it Remark}.
By semi-continuity, one could argue that it suffices to exhibit a single $L$ in each relevant dimension, 
whose stabilizer is made of semisimple elements only. 
An example of such an $L$ can be generated by decomposable tensors $e_p\wedge e_q$, for a set 
$I$ of pairs $(p,q)$. 
If $X$ stabilizes $L$, then for each $(p,q)$ in $I$, $X(e_p\wedge e_q)=Xe_p\wedge e_q+e_p\wedge Xe_q$ must be a combination of the $e_r\wedge e_s$, for $(r,s)$ in $I$. As a consequence, $Xe_p$ must be a linear combination of the $e_r$'s such that $(r,q)$ belongs to $I$. In other words, $X_{rp}=0$ as soon as  
there exists
a $q\ne r,p$ such that $(p,q)\in I$ but $(r,q)\notin I$. We can then look for configurations 
$I$ such that for any pair $p\ne r$, there exists $q$ satisfying this property; then any $X$ in the 
stabilizer of $L$ will have to be diagonal in our fixed basis.

A direct verification shows that we can choose $$I=\{(13),(16),(25),(34),(45)\},$$ and its
unions with $(26)$ and $(56)$. Since obviously the complements of these sets also satisfy 
the required property, we get a suitable $L$ for each dimension beteween $5$ and $10$.

%In order to act trivially on a 
%eight-dimensional subspace of $\wedge^2\CC^6$, $\bar{g}$ must have Jordan blocks of size $(21111)$
%(or be trivial). Otherwise said, $\bar{g}=Id+u$ with $u=\phi\otimes e$ nilpotent of square zero 
%($\phi(e)=0$). Then the subspace of $\wedge^2\CC^6$ on which $\bar{g}$ acts trivially is 
%$\wedge^2(\ker\phi)+e\wedge V_6$, which has dimension $11$ and must contain $U_8$. If we count 
%parameter, we have $5$ for $\ker\phi$, then $4$ for $[e]$, then $8\times 3=24$ for $U_8$. 
%Hence a total of $33$ parameters, much smaller than the dimension $56$ of $G(8,15)$.

\subsection{Semisimple elements}
Let $g$ be a semisimple element in $\GL_6$, with eigenvalues
$t_1,\ldots , t_6$. The codimension in $\GL_6$ of the orbit of $g$ is
\[
  \mathrm{codim} (\O_g) = \tsum_s n_s(g)^2,
\]
where the $n_s(g)$ are  the
multiplicities of the eigenvalues.
The Weyl group is the symmetric group $\Sn[6]$,  acting by
permutations on the $t_i$'s. 
The eigenvalues of the action of $g$ on $\wedge^2\CC^6$ are the
$t_it_j$ for $1\le i<j\le 6$, each with multiplicity $1$.  These
eigenvalues are not always distinct, and we shall discuss their
possible collapsings as explained in \S \ref{strategy-smspl}.

\subsubsection{Regular case} Assume $g$ is regular, i.e., its
eigenvalues $t_i$ are pairwise distinct. Then the conjugacy class of
$g$ has dimension $30$.
The eigenvalue $t_it_j$ can coincide with $t_kt_l$ only if the pairs
$(i,j)$ 
and $(k,l)$ are disjoint. As a consequence, each eigenspace $E_\mu$
for the action on $\wedge^2 \CC^6$ has dimension at most three.
We claim that then
\eqref{OK} holds unless $m=3$ or $12$ and there are at least three
$3$-dimensional $E_\mu$'s, as follows.
This is rather easy to see, but we may just as well use our python
arsenal: we (i) write down all partitions of $15$ as sums of integers
not larger than $3$, then (ii) select those elements in the list
maximal with respect to the partial order indicated in
Remark~\ref{rk:monotone}, and eventually (iii)
compute for all of them the maximum in
Proposition~\ref{dim-stable}~(1) for all $m=3,\ldots,12$.
We find out that \eqref{OK} holds unless $m=3$ or $12$, and in the
latter two cases we get the list of the possible cases in which it is
violated, 
from which we see that three $3$-dimensional eigenspaces are
needed. 
All this is transcripted in the output file
\texttt{g08\_smspl\_pyout.txt}.
(In practice one may skip stage (ii) without any trouble, but then the
output becomes artificially much longer and this is a little
unpleasant). 

Let us decide whether it is indeed possible to have three
$3$-dimensional eigenspaces. 
Up to acting with the Weyl group, we may assume that a first triple
collapsing involved is
\[
  12 = 34 = 56,
  \text{\ i.e.,\ }
  t_1t_2 = t_3t_4 = t_5t_6.
\]
Again up to the Weyl action, the second one is necessarily
either (i) $13 = 2* = **$,
or (ii) $13 = 5* = **$
(i.e., either $t_1t_3 = t_2t_a = t_bt_c$,
or $t_1t_3 = t_5t_a = t_bt_c$).
In case (i), we may end up with either
$13=24=**$ or $13=25=**$.
In case (ii), we may end up with 
$13=52=**$ which has already been found,
since $13=56=**$ is impossible as $56$ is already involved in the
first triple collapsing. The two possibilities left are thus
$13=24=56$ which in fact is impossible, and $13=25=46$.
The upshot is that the only possibility to have two triple collapsings
is
\[
  12 = 34 = 56
  \text{ and } 13=25=46.
\]
One finds (for instance using Macaulay2, see the ancillary files
listed in \S \ref{s:anciliary})
% \texttt{g08-collapse.m2})
that the only possibility with all $t_i \neq 0$
and pairwise distinct is
\[
  (t_1,t_4,t_5)=(a,aj,aj^2)
  \text{ and }
  (t_2,t_3,t_6)=(b,bj^2,bj)
\]
with $a,b\in \CC^*$ and $j$ a primitive cubic root of $1$.
This involves exactly three triple collapsings, the third one being
$16=24=35$.
\gag\label{g8.1}

\subsubsection{Subregular case}
By this we mean that the eigenvalues of $g$ have multiplicities at
most equal to $2$.
We shall examine successively the three possibilities
$[2,1^4]$, $[2^2,1^2]$ and $[2^3]$.

\smallskip
A) $t_1\,(2),\ t_3,t_4,t_5,t_6\,(1)$.
By this we mean that $g$ has the double eigenvalue $t_1$, and the four
simple eigenvalues $t_3,t_4,t_5,t_6$, all five pairwise distinct.
Then the conjugacy class of $g$ has dimension $28$.
The eigenvalues for the representation are
\[
\begin{array}{rlrlrl}
  (i) &  t_1^2 \ (1) & (ii) & t_1t_3 \ (2)
                     & (iii) & t_3t_4, t_3t_5, t_3t_6 \ (1) \\
      &&& t_1t_4 \ (2) && t_4t_5, t_4t_6 \ (1) \\
      &&& t_1t_5 \ (2) && t_5t_6 \ (1) \\
      &&& t_1t_6 \ (2)
\end{array}
\]
The possible collapsings are the following. The eigenvalue $t_1^2$ is
necessarily distinct from those of type (ii), but may equal some of
type (iii). Eigenvalues of type (ii) are pairwise distinct, and each
may equal at most one of type (iii) (for instance, $t_1t_3$ may equal
only one among $t_4t_5, t_4t_6, t_5t_6$). Eigenvalues of type (iii)
may collapse at worst in pairs.

It follows that the eigenspaces in $\wedge^2\CC^6$ always have
dimension at most $3$, hence \eqref{OK} holds in all cases
(including $m=3$, since conjugacy classes now have dimension only
$28$)
by the analysis carried out in the regular case.

\smallskip
B) $t_1,t_3\,(2),\ t_5,t_6\,(1)$.
Then the conjugacy class of $g$ has dimension $26$ and the eigenvalues
in the representation are
\[
\begin{array}{rlrlrlrl}
  (i) & t_1^2 \ (1) & (ii) & t_1t_3 \ (4) & (iii) & t_1t_5 \ (2)
  & (iv) & t_5t_6 \ (1) \\
  & t_3^2 \ (1) &&&&  t_1t_6 \ (2) \\
  &&&&& t_3t_5 \ (2) \\
  &&&&& t_3t_6 \ (2) 
\end{array}
\]
The possible collapsings are the following. 
The eigenvalue $t_1^2$ may equal $t_3^2$, is different from $t_1t_3$,
may equal at most one type (iii) eigenvalue, or the type (iv)
eigenvalue. The eigenvalue $t_1t_3$ may only collapse with
$t_5t_6$. Type (iii) eigenvalues may at most collapse in pairs, and
all are distinct from $t_5t_6$.

Assume the type (iii) eigenvalues collapse in two pairs. Up to the
Weyl action, we may assume that $15=36$ and $16=35$.
Since $t_1,t_3,t_5,t_6$ are pairwise distinct and nonzero, one finds
that necessarily $t_3=-t_1$ and $t_6=-t_5$.
In this case we find the eigenvalues
\[
\begin{array}{rrrr}
  t_1^2 \ (2) & -t_1^2 \ (4) & t_1t_5 \ (4) & -t_5^2 \ (1) \\
              &&  -t_1t_5 \ (4)
\end{array}
\]
with the only possible further collapsing $t_1^2=-t_5^2$,
and \eqref{OK} always holds:
by monotonicity it suffices to consider the case with the largest
possible eigenspaces, which we have just found to be $[4^3,3]$, and
for which we verify \eqref{OK} for all $m=3,\ldots,12$ using our
python toolkit.

If $t_1t_3=t_5t_6$, there can be at most one collapsing among type
(iii) eigenvalues (if there are two, we must be in another regularity
stratum), say $t_1t_5=t_3t_6$, and then
$t_6=-t_1$ and $t_5=-t_3$,
so we find the eigenvalues
\[
\begin{array}{llrr}
  t_1^2 \ (1) & t_1t_3 \ (5) & -t_1t_3 \ (4) \\
  t_3^2 \ (1) &&  -t_1^2 \ (2) \\
  && -t_3^2 \ (2) 
\end{array}
\]
with the only possible further collapsing $t_1^2=-t_3^2$,
so that the case with the largest possible eigenspaces is
$[5,4,3^2]$. In this case \eqref{OK} holds for all $m=3,\ldots,12$,
and so does it for smaller eigenspaces by monotonicity.

In the other cases we have smaller eigenspaces, hence \eqref{OK} holds
in all these cases as well, again by monotonicity.

\smallskip
C) $t_1,t_2,t_3\,(2)$.
Then the conjugacy class of $g$ has dimension $24$ and the eigenvalues
in the representation are
\[
\begin{array}{rlrl}
  (i) & t_1^2 \ (1) & (ii) & t_1t_2 \ (4) \\
  & t_2^2 \ (1) && t_1t_3 \ (4) \\
  & t_3^2 \ (1) && t_2t_3 \ (4) \\
\end{array}
\]
The possible collapsings are the following.
There can be one simple collapsing between type (i) eigenvalues.
Type (ii) eigenvalues are always pairwise distinct. Each type (i)
eigenvalue may equal exactly one type (ii) eigenvalue.

If a collapsing between type (i) eigenvalues occurs, say
$t_1^2=t_2^2$, then the only other possible collapsing is
$t_3^2=t_1t_2$, in which case the eigenspaces give a partition
$[5,4^2,2]$, and \eqref{OK} holds for all $m$ in the relevant
range (again we consider the case $[5,4^2,2]$ using our python toolkit
and conclude by monotonicity).

If two collapsings between type (i) and (ii) eigenvalues occur,
say $t_1^2=t_2t_3$ and $t_2^2=t_1t_3$, then one finds that necessarily
\[
  (t_1,t_2,t_3)=(a, aj,aj^2)
\]
for some $a \in \CC^*$ and $j$ a primitive cubic root of unity,
so that the third relation $t_3^2=t_1t_2$ also holds, and we have
three $5$-dimensional eigenspaces.
\eqref{OK} holds only for $4 \leq m \leq 11$, whereas
for $m=3$ (resp.\ $12$) we find a $36$-dimensional family of stable
$3$-spaces by 
considering sums of three lines in the three $5$-dimensional
eigenspaces (resp.\ the dual configuration).
\gag\label{g8.2}

\subsubsection{Penultimate cases}
We now consider the cases in which $g$ has one eigenvalue of
multiplicity at least $3$, and in total at least three pairwise
distinct eigenvalues,
which amount to the following partitions of
$6$: $[3,1^3]$,
$[3,2,1]$,
and $[4,1,1]$.

\smallskip
A) $t_1\,(3),\ t_2,t_3,t_4\,(1)$.
Then the conjugacy class of $g$ has dimension $24$.
The eigenvalues for the representation are
\[
\begin{array}{rlrlrl}
  (i) & t_1^2 \ (3) & (ii) & t_1t_2 \ (3) & (iii) & t_2t_3 \ (1) \\
      &             &      & t_1t_3 \ (3) && t_2t_4 \ (1) \\
      &             &      & t_1t_4 \ (3) && t_3t_4 \ (1) \\
\end{array}
\]
Type (ii) eigenvalues are necessarily pairwise distinct, and so are
those of type (iii).
The eigenvalue $t_1^2$ must be distinct from those of type (ii), and
may equal at most one of type (iii).
Each eigenvalue of type (ii) may coincide with only one eigenvalue of
type (iii).
So at most the eigenspaces give the partition $[4^3,3]$ which has
already been considered in the subregular case, and
\eqref{OK} holds in all cases.

\smallskip
B) $t_1\,(3),\ t_2\,(2),\ t_3\,(1)$.
Then the conjugacy class of $g$ has dimension $22$.
The eigenvalues for the representation are
\[
\begin{array}{rlrl}
  (i) & t_1^2 \ (3) & (ii) & t_1t_2 \ (6) \\
      & t_2^2 \ (1) && t_1t_3 \ (3) \\
      &             && t_2t_3 \ (2)
\end{array}
\]
Type (ii) eigenvalues are necessarily pairwise distinct.
Each eigenvalue of type (i) may coincide with only one eigenvalue of
type (i), and $t_1^2$ may equal $t_2^2$.
If $t_1^2=t_2^2$, then no further collapsing is possible, and the
eigenspaces give the partition
$[6,4,3,2]$. In this case \eqref{OK} holds for all $m$.
Otherwise, it may happen at worst that $t_1^2=t_2t_3$ and
$t_2^2=t_1t_3$, in which case we would have the partition
$[6,5,4]$, and we find that \eqref{OK} holds in all cases.

\smallskip
C) $t_1\,(4),\ t_2,t_3\,(1)$.
Then the conjugacy class of $g$ has dimension $18$.
The eigenvalues for the representation are
\[
\begin{array}{rlrl}
  (i) & t_1^2 \ (6) & (ii) &  t_1t_2,t_1t_3 \ (4) \\
      &             &      & t_2t_3 \ (1) \\
\end{array}
\]
The only possible collapsing is $t_1^2=t_2t_3$, so at worst we get the
partition $[7,4^2]$, and \eqref{OK} holds in all cases.

\subsubsection{Remaining cases}
Eventually, we consider the cases in which $g$ has two distinct
eigenvalues $t_1,t_2$ of multiplicities $\mu_1,\mu_2$ respectively.
Then the eigenvalues for the representation are
\[
\begin{array}{lll}
  t_1^2 \ \binom {\mu_1} 2
  & t_2^2 \ \binom {\mu_2} 2
  & t_1t_2 \ (\mu_1\mu_2)
\end{array}
\]
with the only possible collapsing $t_1^2=t_2^2$.

\smallskip
A) If $(\mu_1,\mu_2)=(5,1)$, then the conjugacy class of $g$ has
dimension $10$, and the only possible partition is $[10,5]$.
\eqref{OK} holds for all $m$.

\smallskip
B) If $(\mu_1,\mu_2)=(4,2)$, then the conjugacy class of $g$ has
dimension $16$, and at worst we get the partition $[8,7]$.
\eqref{OK} holds for all $m$.

\smallskip
C) If $(\mu_1,\mu_2)=(3,3)$, then the conjugacy class of $g$ has
dimension $18$.

If $t_1^2 \neq t_2^2$ we get the partition
$[9,3^2]$. We obtain a $36$-dimensional family of pairs $(L,g)$ with
$g.L=L$ and $\dim (L)=3$ (resp.\ $12$) by considering those $L$
entirely contained 
in the $9$-dimensional eigenspace (resp.\ the dual configuration).
\gag\label{g8.3}

If  $t_1^2 = t_2^2$ we get the partition
$[9,6]$. We obtain a $37$-dimensional family of pairs $(L,g)$ with
$g.L=L$ and $\dim (L)=3$ (resp.\ $12$) by considering sums of a
$2$-plane in the 
$9$-dimensional eigenspace and a $1$-plane in the $6$-dimensional
one (resp.\ the dual configuration).
\gag\label{g8.4}
\\(We also get a $36$-dimensional family of pairs $(L,g)$ with $L$
entirely contained in the $9$-dimensional eigenspace (resp.\ in the
dual configuration, as a degenerate instance of case \ref{g8.3}
above). 

\subsubsection{Conclusion}
We have found that if $L$ is a generic $k$-plane with
$3<k<12$, then its stabilizer is trivial, whereas if $L$ is generic of
dimension $3$ or $12$, then its stabilizer may contain only elements as
described in the four cases
\ref{g8.1}, \ref{g8.2}, \ref{g8.3}, \ref{g8.4}
(see below for an explicit description).

\subsection{Codimension three} 

We consider in this section a general three-dimensional subspace
$L\subset \wedge^2\CC^6$
(later on we shall consider $L^\perp$ which has the dual size).
By the previous analysis, the stabilizer $S_L$ of $L$ in $\PSL_6$ (not $\SL_6$) is made of semisimple elements, and it can contain
\begin{enumerate}
\item at most a one dimensional family (not necessarily connected a
priori) of elements with two eigenvalues of multiplicity three, such
that if $A$ and $B$ denote the two eigenspaces, then $L\subset
A\otimes B$ (case \ref{g8.3});
\item at most a one dimensional family of involutions with two
eigenspaces $E,F$ of dimension three, such that $L$ is the sum of a
line $L_1\subset \wedge^2E\oplus \wedge^2F$ and a plane $L_2\subset
E\otimes F$
(case \ref{g8.4});
\item at most a one dimensional family of elements with eigenvalues
$a, ja, j^2a, a^{-1}, ja^{-1}, j^2a^{-1}$ for some $a\in \CC^*$, with
$a^6\ne 1$; then the induced action on $\wedge^2\CC^6$ has the
eigenvalues $1,j,j^2$ with multiplicity three, and $L$ is the sum of
three lines in those eigenspaces
(case \ref{g8.1});
\item a finite number of elements with eigenvalues $1,j,j^2$, of
multiplicity two; then the induced action on $\wedge^2\CC^6$ has the
eigenvalues $1,j,j^2$ with multiplicity five, and $L$ is the sum of
three lines in those eigenspaces
(case \ref{g8.2}).
\end{enumerate}
We will call these elements of type (1) to (4).

\medskip\noindent {\bf Type (1).} Let us first explain the origin of the first family. 

\begin{lemma} There exists a unique pair $(A,B)$ of transverse three-planes in $\CC^6$ such that 
$L\subset A\otimes B\subset \wedge^2\CC^6$.
\end{lemma} 

\begin{proof} Observe that once we know that the pair $(A,B)$ does exist, the elements $g_s=s\Id_A+s^{-1}\Id_B$ belong to 
the connected component $S_L^0$ of $S_L$. Since we know that this connected component is at most one dimensional, it must 
coincide with the set of those elements. In particular the pair $(A,B)$ must be unique. 
%Note also that $S_L^0$ not only preserves $L$, but acts trivially on $L$. 

\smallskip 
In order to prove the existence of the pair $(A,B)$, we use the following approach. Let $U\subset \G(3,6)\times \G(3,6)$ be the open subset 
of pairs of transverse planes $(A,B)$, and $ Z\rightarrow U$ be the relative Grassmannian with fiber $\G(3,A\otimes B)$.
The dimension of $Z$ is $36$. We need to prove that the natural map $\pi: Z\rightarrow \G(3, \wedge^2\CC^6)$ that forgets the pair
$(A,B)$ is dominant, hence generically finite since the dimensions coincide. For this we will prove that the differential 
of $\pi$ is an isomorphism at the general point $z=(A,B,L)$ of $Z$. Recall that the tangent space to a Grassmannian 
is the bundle of morphisms from the tautological to the quotient vector bundle. We readily deduce that 
the tangent space of $Z$ at $z$ fits into the relative exact sequence 
$$0\to \Hom(L,A\otimes B/L)\to  T_zZ\to  \Hom(A,B)\oplus \Hom(B,A)\to  0.$$
Moreover $\wedge^2\CC^6=\wedge^2A\oplus A\otimes B\oplus\wedge^2B$, so 
$$T_L\G(3, \wedge^2\CC^6)\simeq \Hom(L,\wedge^2A\oplus (A\otimes B/L)\oplus\wedge^2B).$$
We are therefore reduced to showing that the morphisms
$$\sigma : \Hom(A,B)\longrightarrow \Hom(L,\wedge^2B), \qquad \tau : \Hom(B,A)\longrightarrow \Hom(L,\wedge^2B)$$
are isomorphisms, where $\sigma$ is defined by sending $u\in
\Hom(A,B)$ to the composition
$$L\hookrightarrow A\otimes B\xrightarrow{u\otimes \Id_B}
B\otimes B\longrightarrow \wedge^2B,$$
and $\tau$ is defined similarly. The following Lemma therefore concludes the proof of the previous one. 
\end{proof}

\begin{lemma}\label{normal} Let $a_1,a_2,a_3$ be some basis of $A$, 
and $b_1,b_2,b_3$ some basis of $B$. Consider the $3$-space $L$ generated by 
\[
\begin{array}{rl}
p &= xa_1\otimes b_1+ya_2\otimes b_2+za_3\otimes b_3,  \\ 
q &= za_1\otimes b_2+xa_2\otimes b_3+ya_3\otimes b_1,  \\
r &= ya_1\otimes b_3+za_2\otimes b_1+xa_3\otimes b_2,
\end{array}
\]
for $[x,y,z]$ in $\PP^2$ such that $x^3, y^3, z^3$ are pairwise distinct. 
Then $\pi$ is {\'e}tale at $z=(A,B,L)$.
\end{lemma}

In particular, $L$ is generic in $\G(3,\wedge^2\CC^6)$. As a consequence it is also generic in $\G(3,A\otimes B)$. 

\begin{proof} We make an explicit computation.
% Let $\ell_1, \ell_2,\ell_3$ denote the basis vectors of $L$ given in
% the Lemma.
The map that we must check to be an isomorphism sends $u\in  \Hom(A,B)$ to the morphism from $L$ to $\wedge^2B$ defined by 
$$\begin{array}{rcl}
p & \mapsto & xu(a_1)\wedge b_1+yu(a_2)\wedge b_2+zu(a_3)\wedge b_3,  \\ 
q & \mapsto & zu(a_1)\wedge b_2+xu(a_2)\wedge b_3+yu(a_3)\wedge b_1,  \\
r & \mapsto & yu(a_1)\wedge b_3+zu(a_2)\wedge b_1+xu(a_3)\wedge b_2.
\end{array}$$
Denote $u(a_i)=\sum u_{ij}b_j$ and suppose that $u$ is mapped to the zero morphism. Then we get nine equations on the $u_{ij}$'s, 
which split into three subsystems of size three. For instance, the three equations involving $u_{11}, u_{23}, u_{32}$ are 
$$yu_{23}-zu_{32}=zu_{11}-yu_{32}=yu_{11}-zu_{23}=0,$$
and this system is invertible if and only if $y^3-z^3\ne 0$. This implies the claim.
\end{proof}

\medskip\noindent {\bf Type (2).} 
Now consider an element $h$ of type (2) in $S_L$. Since any $g_s=s\Id_A+s^{-1}\Id_B$ belongs to the connected component of $S_L$, 
the product $g_sh$ must remain of type (2); in particular it must be an involution. We deduce that $h$ must exchange $A$ and $B$. 

We claim that such elements do exist. In particular the natural map $S_L\rightarrow\Sn[2]$, the permutation group 
of the pair $A,B$, is surjective. In order to see this, we may suppose that $L$ is given in the normal form of Lemma \ref{normal}.
Then we can exhibit the following 
type (2) transformations preserving $L$: just fix an integer $k$ (modulo $3$), choose  $\zeta$ some root of unity, 
and let $h(a_i)=\zeta^ib_{i+k}$ (where indices are computed modulo $3$).

\medskip\noindent {\bf Type (3).} 
Consider an element $g$ in $S_L$ whose eigenvalues are $a, ja, j^2a, a^{-1}, ja^{-1}, j^2a^{-1}$
for some $a\in \CC^*$. Denote by $e_1,e_2,e_3,f_1,f_2,f_3$ a basis of eigenvectors for these eigenvalues.   
The action of $g$ on $\wedge^2\CC^6$ admits the eigenvalues $1,j,j^2$, each with multiplicity $3$ (except for special
values of $a$), and $L$ must be a direct sum of lines $L_1$, $L'_1$, $L''_1$ contained in the associated eigenspaces,
that is 
$$L_1\subset \langle e_1\wedge f_1, e_2\wedge f_3, e_3\wedge f_2\rangle ,\qquad 
L'_1\subset \langle e_1\wedge f_2, e_2\wedge f_1, e_3\wedge f_3\rangle ,$$
$$L''_1\subset \langle e_1\wedge f_3, e_2\wedge f_2, e_3\wedge f_1\rangle .$$
In particular $L$ is contained in $\langle e_1,e_2,e_3\rangle\otimes  \langle f_1,f_2,f_3\rangle$, and by the previous Lemma
we may suppose that $A=\langle e_1,e_2,e_3\rangle$ and $B=  \langle f_1,f_2,f_3\rangle$. In particular the cube of $g$ 
acts on $A$ and $B$ by homotheties, and belongs to the connected component of $S_L$.

\medskip\noindent {\bf Type (4).} Finally, consider an element $h$ of type (4). Then $g_shg_s^{-1}$ is also of type (4). 
Since we have only finitely many such elements in $S_L$, and $s$ varies in a connected set, we conclude that necessarily, 
$g_shg_s^{-1}=h$. In particular the eigenspaces $C_1,C_2,C_3$ of $h$ are direct sums of their intersections with $A$ and $B$. 
We claim that each of them must be the sum of a line in $A$ and a line in $B$. Indeed, if this were not the case, we would 
be able to deduce that, up to some permutation of indices, $C_1$ is contained in $A$, $C_3$ is contained in $B$, and $C_2$ is the 
sum of a line $a\subset A$ with a line $b\subset B$. Recall that $L$ must be the direct sum of lines $L_1\subset \wedge^2C_1\oplus 
C_2\otimes C_3$, $L'_1\subset \wedge^2C_2\oplus C_1\otimes C_3$, $L''_1\subset \wedge^2C_3\oplus 
C_1\otimes C_2$. Since we also know that $L\subset A\otimes B$, we would deduce that in fact $L_1=a\otimes C_3$, $L''_1=C_1\otimes b$ 
and $L'_1\otimes a\otimes b\oplus C_1\otimes C_3$. Counting dimensions, we would conclude that $L$ cannot be general. 

So the conclusion is that $h$ preserves $A$ and $B$, and its
restriction to  each of these has eigenvalues $1,j,j^2$. In particular its cube 
has to belong to $S_L^0$.

\medskip\noindent {\bf Synthesis.} By the previous analysis, the
action of $S_L$ on $\PP (L)$ induces an injective morphism of
$S_L/S^0_L$ into $\PSL(L)$. The induced action on $\PP(L)$ preserves
the genus one curve $\mathcal{C}$ cut out on $\PP (L)$ by the Pfaffian
hypersurface.

Let $T_L\subset S_L$ denote the subgroup of elements sending $A$ and
$B$ to themselves, hence of type (1), (3) or (4).  We have seen that
the image of $T_L$ in $\PSL{}(L)$ consists of regular semisimple
elements whose cubes are all trivial; this implies that they act on
$\mathcal{C}$ (which is a general curve of genus $1$) by translation
by some $3$-torsion point.

Elements of type (2), that is, in $S_L-T_L$, induce involutions of
$\mathcal{C}$ that must be point reflections across an inflection
point. Indeed, recall that an element of type (2) in $S_L$ has
eigenspaces $E,F$ of dimension three, such that $L$ is the sum of a
line $L_1\subset \wedge^2E\oplus \wedge^2F$ and a plane $L_2\subset
E\otimes F$.  In this situation, $L_1$ is generated by a two-form
$\lambda$ of rank four, hence degenerate, and moreover
% for degree reasons we have $\lambda\wedge L_2\wedge L_2=0$.
$L_1 \wedge L_2\wedge L_2=0$.
Therefore, if $\mu, \nu$ are
two-forms that generate $L_2$, the Pfaffian of a general
element of $L$ writes
$$\Pf(x\lambda+y\mu+z\nu)=
3x^2\lambda^2\wedge (y\mu+z\nu)
+\Pf(y\mu+z\nu).$$  
This shows that $p=[1,0,0]=[L_1]$ is an inflection point of
$\mathcal{C}$. 
Morever, the line $x=0$ cuts (in general) the curve $\mathcal{C}$ at
three points $q_1, q_2, q_3$, such that the tangent line to
$\mathcal{C}$ at each $q_i$ passes through $p$,
which means that the degree zero divisors $p-q_i$ are $2$-torsion.
Therefore $q_1, q_2, q_3$ are fixed points of the point reflection
across $p$.
The upshot is that the involution of $\PP(L)$
associated to the decomposition $L = L_1 \oplus L_2$,
once restricted to $\mathcal{C}$, is just the symmetry
with respect to the inflection point $p=[L_1]$.

So far, we have proved that the image $S_L/S_L^0$ injects in the
subgroup $(\ZZ/3\ZZ)^2\rtimes (\ZZ/2\ZZ)$ of automorphisms of
$\mathcal{C}$ of translations by an element of $3$-torsion and
point reflections across an inflection point. Let us show that this
injection is in fact an isomorphism.
In our analysis of elements of type (2) above, we have already seen
that the image of $S_L/S_L^0$ indeed contains elements of order
$2$. It will thus be sufficient to prove that all order $3$
translations are in the image.

To do so we may assume that $L$ is given in the normal form of
Lemma~\ref{normal}. Then the curve $\mathcal{C}$ is defined by the
cubic polynomial
\[
  \Pf(up+vq+wr)
  = \begin{vmatrix}
    ux & vz & wy \\
    wz & uy & vx \\
    vy & wx & uz
  \end{vmatrix}
  = xyz (u^3+v^3+w^3) -(x^3+y^3+z^3) uvw,
\]
and
its group $(\ZZ/3\ZZ)^2$ of translations by $3$-torsion points is
generated by the two transformations
\[
  s: (u:v:w) \mapsto (u:jv:jw)
  \quad \text{and} \quad
  t: (u:v:w) \mapsto (w:u:v)
\]
(we leave this as an entertaining exercise; hints may be found in
\cite[\S 3.1]{dolgachev}).
The translation $s$ is realized by the type (3) element $g$ acting
diagonally on the basis 
$(a_1,a_2,a_3,b_1,b_2,b_3)$ with eigenvalues
$\alpha, j\alpha, j^2\alpha, \alpha^{-1}, j\alpha^{-1}, j^2\alpha^{-1}$
for some $\alpha\in \CC^*$.
The translation $t$ is realized by the type (4) element $g$ sending
$a_i$ to $a_{i+1}$ and 
$b_i$ to $b_{i-1}$ for $i=1,2,3$, with indices computed modulo $3$.

We have thus proved that the image of $S_L/S_L^0$ contains a set of
generators of $(\ZZ/3\ZZ)^2\rtimes (\ZZ/2\ZZ)$, and
in conclusion we have proved the following.

\begin{prop} $\Aut_G(L)\simeq S_L/S_L^0\simeq  (\ZZ/3\ZZ)^2\rtimes (\ZZ/2\ZZ)$. \end{prop} 

\smallskip
Note in particular that $\Aut_G(L)$ contains eight elements of order three and nine involutions.

\medskip On the dual size, observe that the connected component $S_L^0$ does not act trivially on $\PP (L^\perp)$. 
Indeed, $L^\perp$ is the sum of $\wedge^2A^\vee$,  $\wedge^2B^\vee$, and a dimension three subspace of 
$A^\vee\otimes B^\vee$, and a non trivial element $g_s=s\Id_A+s^{-1}\Id_B$ of $S_L^0$ acts on these 
pieces with distinct eigenvalues
(by abuse of notation, we denote by $A^\vee$ the subspace $B^\perp$ of
$(\CC^6)^\vee$, and similarly $A^\perp$ by $B^\vee$). We get
$$\Aut_G(L^\perp)\simeq S_L\simeq  \CC^*\rtimes ((\ZZ/3\ZZ)^2\rtimes
(\ZZ/2\ZZ)).$$
That this is indeed a semi-direct product comes from the fact that we
have from the above analysis an explicit splitting of the exact
sequence
\[
  0 \to
  \CC^* \to
  \Aut_G(L^\perp) \to
  ((\ZZ/3\ZZ)^2\rtimes (\ZZ/2\ZZ))
  \to 0.
\]

The corresponding  codimension three section $X$ of $\G(2,6)$  is a Fano fivefold of index $3$ 
with automorphism group $\Aut(X)=\Aut_G(L^\perp)$. 
The involutions in this automorphism group fix the intersection of 
$X$ with their eigenspaces, that is, the intersection of $\G(2,6)$ with a general hyperplane in $\PP(\wedge^2E\oplus \wedge^2F)^\vee$,
and a general codimension two subspace in $\PP(E\otimes F)^\vee$. The former is the union of two skew lines, and the latter 
is a del Pezzo surface of degree six. 

The order three elements of the automorphism group are of the form $\Id_P+j\Id_Q+j^2\Id_R$, where $P,Q,R$ are 
transverse planes in $\CC^6$. The eigenspaces of the induced action on $\wedge^2\CC^6$ are $\wedge^2P\oplus Q\otimes R$
and its two siblings. Each of them intersect $L^\perp$ along a generic hyperplane. Since 
$$\PP(\wedge^2P\oplus Q\otimes R)\cap G(2,6)=\PP(\wedge^2P)\cup \PP(Q)\times\PP(R),$$
a generic hyperplane section gives a conic. We conclude that the fixed loci of the order three automorphisms of $X$
are unions of three conics.

\medskip\noindent {\it Remark}. As we already mentioned in the Introduction, codimension three sections 
of $\G(2,6)$ were considered before by Piontkowski and Van de Ven \cite{pv}. Through direct computations 
they identified the connected component of the automorphism group, and they proved that the quotient 
embeds in the group of projective transformations of the associated plane cubic curve. Our result, which follows from  a 
completely different proof, is more precise since we completely identify this quotient, as well as the geometric
nature of its eighteen elements.

\section{Genus $10$} 
\setcounter{genus}{10}

We proceed with the case of genus $10$, which is rather straightforward because the relevant Lie group $G_2$ has only rank two. Indeed, 
this implies that the number of unipotent orbits, and the number of cases to be discussed for semisimple elements, are relatively small.

\subsection{Unipotent  elements}
Recall that the root system of $\fg_2$ is as follows, where we denote
by $(\alpha, \beta)$ a pair of simple roots, with $\alpha$ long and
$\beta$ short:
\\

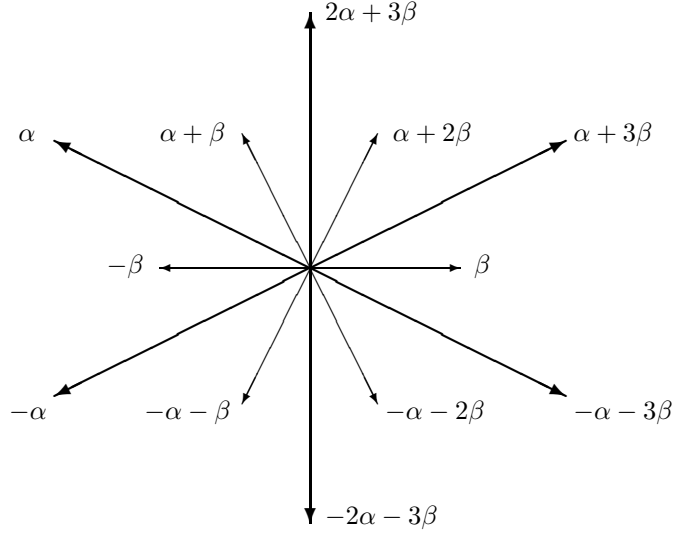
\setlength{\unitlength}{2mm}
\begin{picture}(20,37)(2,2)
\small
\put(30,20){\vector(1,0){10}}
\put(30,20){\vector(1,2){4.5}}
\put(30,20){\vector(1,-2){4.5}}
\put(30,20){\vector(-1,0){10}}
\put(30,20){\vector(-1,2){4.5}}
\put(30,20){\vector(-1,-2){4.5}}

\thicklines
\put(30,20){\vector(0,1){17}}
\put(30,20){\vector(2,1){17}}
\put(30,20){\vector(-2,1){17}}
\put(30,20){\vector(0,-1){17}}
\put(30,20){\vector(2,-1){17}}
\put(30,20){\vector(-2,-1){17}}

\put(31,36.5){$2\alpha+3\beta$}

\put(10.6,28.5){$\alpha$}
\put(20,28.5){$\alpha+\beta$}
\put(35.5,28.5){$\alpha+2\beta$}
\put(47.5,28.5){$\alpha+3\beta$}

\put(16.5,19.7){$-\beta$}
\put(40.9,19.7){$\beta$}

\put(47.5,10){$-\alpha-3\beta$}
\put(35,10){$-\alpha-2\beta$}
\put(19,10){$-\alpha-\beta$}
\put(10,10){$-\alpha$}

\put(31,3){$-2\alpha-3\beta$}
\end{picture}

According to \cite[p.128]{cm}, there are four non trivial nilpotent orbits in $\fg_2$.
They admit  the following representatives, where as usual we have decomposed $\fg_2$ into a Cartan subalgebra 
and root spaces, and $X_\gamma$ denotes a generator of the root space $\fg_\gamma$. 
$$\begin{array}{llll}
\mathrm{Orbit} & \mathrm{Dimension} & \mathrm{Representative} & \mathrm{Jordan\; type} \\   
\mathcal{O}_{reg} & 12 & X_{\alpha} + X_{\beta} & 12,2 \\
\mathcal{O}_{subreg} & 10 & X_{\alpha+2\beta} + X_{\beta}  & 5,3^3\\
\mathcal{O}_{short} & 8 & X_{\beta} & 4^2,3,1^3 \\
\mathcal{O}_{min} & 6 & X_{\alpha} & 3,2^4,1^3
\end{array}$$
The determination of the Jordan types follows from explicit computations done with the help of \cite{wil}. 
(As observed by a referee, one could also include each nilpotent element into an $\fsl_2$-triple and compute
the eigenvalues of the semisimple element in the triple. Indeed these eigenvalues determine the $\fsl_2$-module structure, 
hence the Jordan type of the nilpotent element.) 
Using Proposition \ref{dim-stable} we can then exclude the possibility for a unipotent element of $G_2$ to stabilize a 
generic subspace $L$ of $\fg_2$ of dimension $2$ to $12$.

\subsection{Semisimple elements}

Suppose $g$ is a semisimple element of $G_2$, say an element of our 
fixed maximal torus. Then its eigenvalues in the adjoint
representation on $\mathfrak g_2$ are $1$ with multiplicity $2$, the
rank of $G_2$, and the values taken by the roots, in other words:
\[
  \begin{array}{ll}
  1 \ (2) & \alpha^2\beta^3 \ (1)  \\
         & \alpha, \alpha\beta, \alpha\beta^2, \alpha\beta^3 \ (1) \\
         & \beta^{-1}, \beta \ (1)\\
         &  \alpha^{-1}\beta^{-3}, \alpha^{-1}\beta^{-2},
           \alpha^{-1}\beta^{-1}, \alpha^{-1} \ (1) \\
         & \alpha^{-2}\beta^{-3} \ (1)  \\
  \end{array}
\]

By Remark~\ref{l:dim-conjcl}, the conjugacy class of $g$ has dimension
$\dim(G_2)-\delta$, with $\delta$ the multiplicity of the eigenvalue
$1$ (recall that $\dim (G_2)=14$).

The Weyl group is isomorphic to the dihedral group of order $12$,
generated by the rotation of order $\pi/3$ and the reflection across
the $x$-axis. It acts as such on the roots,  pictured as
above. 

Degenerations occur if one root takes the value $1$ on $g$,
and collapsings occur when two roots take the same value.
In the generic case, $\alpha,\beta \neq 1$ and the roots take pairwise
distinct values, there is only one double
eigenvalue (namely $1$), and the conjugacy class of $g$ has dimension
$12$, so that \eqref{OK} always holds in  dimensions $2$ to $12$:
this is easily seen, but for completeness we also examine the
eigenspace decomposition $[2,1^{12}]$ with our python procedure.

\subsubsection{Degenerate case}

In this case, we assume that some root takes the value $1$.
Up to the action of the Weyl group we may suppose
that this root is either $\alpha$ or $\beta$. We treat the two cases
separately. 

\smallskip
A) $\alpha=1$. Then we find the eigenvalues
\[
  \begin{array}{lllllll}
    \beta^{-3}\,(2)
    & \beta^{-2} \, (1)
    & \beta^{-1} \, (2)
    & 1 \, (4)
    & \beta \, (2)
    & \beta^{2} \, (1)
    & \beta^{3}\,(2)
  \end{array}
\]
In the generic case, the conjugacy class has dimension $10$.
%The case $\beta=1$ is trivial. 
Collapsings occur if
$\beta$ is a primitive root of $1$ of order $o=2,3,4,5,6$.
\begin{mycpctenum}[i)]
\item if $o=2$, the conjugacy class has dimension $8$,
  and the partition is $[8,6]$;
\item if $o=3$, the conjugacy class has dimension $6$,
  and the partition is $[8,3,3]$;
\item if $o=4,5,6$, the conjugacy class has the generic dimension
  $10$, and the partition is
  $[4^3,2]$,
  $[4,3^2,2^2]$,
  and $[4^2, 2^2,1^2]$
  respectively.
\end{mycpctenum}
We examine all these cases one by one with our python procedure:
\eqref{OK} holds for all relevant $m$ for all of them.

\smallskip
B) $\beta=1$. Then we find the eigenvalues
\[
  \begin{array}{lllll}
    \alpha^{-2}\, (1)
    & \alpha^{-1} \, (4)
    & 1\, (4)
    & \alpha \, (4)
    & \alpha^2 \, (1)
  \end{array}.
\]
In the generic case, the conjugacy class has dimension $10$.
Collapsings occur if
$\alpha$ is a primitive root of $1$ of order $o=2,3,4$.
\begin{mycpctenum}[i)]
\item if $o=2$, the conjugacy class has dimension $8$,
  and the partition is $[8,6]$;
\item if $o=3$, the conjugacy class has the generic dimension $10$,
  and the partition is $[5,5,4]$;
\item if $o=4$, the conjugacy class has the generic dimension
  $10$, and the partition is
  $[4^3,2]$.
\end{mycpctenum}
We examine all these cases one by one with our python procedure:
\eqref{OK} holds for all relevant $m$ for all of them.

\subsubsection{Nondegenerate cases}
It remains to consider those cases in which two roots coincide.
Up to the action of the Weyl group, this reduces to a short list of
possibilities.

If two short roots collapse, we may assume that $\beta$ collapses with
$\beta^{-1},\alpha\beta,\alpha\beta^2$.
In the latter two cases we have respectively $\alpha,\alpha\beta$
which take the value $1$, and these possibilities have already been
investigated. 

If a short root collapses with a long root, we may assume that $\beta$
collapses with $\alpha,\alpha^2\beta^3,\alpha\beta^3$.
We discard the latter possibility, which corresponds to
the degeneration $\alpha\beta^2=1$ given by a short root. We also discard the 
second one which is equivalent, up to the Weyl action, to
$\beta=\beta^{-1}$. 

If two long roots collapse, we may assume that $\alpha$ collapses with
$\alpha^{-1},\alpha^2\beta^3,\alpha\beta^3$. We discard the collapsing
with $\alpha^2\beta^3$ which coincides with a degeneration given by a long root. 
We are thus left with the following list:  a) $\beta=\beta^{-1}$; b) 
$\beta=\alpha$; c) $\alpha=\alpha^{-1}$; d) $\alpha=\alpha\beta^3$.
%\begin{mycpctenum}[a)]
%\item $\beta=\beta^{-1}$;
%\item $\beta=\alpha$;
%\item $\beta=\alpha^2\beta^3$;
%\item $\alpha=\alpha^{-1}$;
%\item $\alpha=\alpha\beta^3$.
%\end{mycpctenum}
%Among these we note that $\beta=\alpha^2\beta^3
%\iff (\alpha\beta)^2=1$ is the same up to the Weyl action as
%$\beta=\beta^{-1} \iff \beta^2=1$.
%We shall now consider separately the four remaining cases.

From now on we exclude  $\alpha=1$ or $\beta=1$, which have already been considered.
We record once and for all that for the partitions
$[5,5,4]$ and $[4^3,2]$, \eqref{OK} holds for all $m$ unconditionally
since conjugacy classes have dimension at most $12$:
these two cases are checked with our usual procedure using our python
toolkit. 

\smallskip
A) $\alpha^2=1$.
We thus have $\alpha = -1$, 
which gives the eigenvalues
\[
  \begin{array}{rrrrrrr}
    -\beta^{-3}\, (1)
    & -\beta^{-2} \, (1)
    & -\beta^{-1} \, (1)
    & -1\, (2)
    & -\beta \, (1)
    & -\beta^2 \, (1)
    & - \beta^3 \, (1)
    \\
    \beta^{-3}\, (1)
    && \beta^{-1} \, (1)
    & 1 \, (2)
    & \beta \, (1)
    && \beta^{3} \, (1)
  \end{array}.
\]
In the generic case, the conjugacy class of $g$ has dimension $12$.
Further collapsings occur if $\beta$ is a
$2,3,4,5,6$-th root
of $1$ or $-1$.
\begin{mycpctenum}[i)]
\item\label{g10--1-1}
  if $\beta=-1$, we find the two eigenvalues
  $1 \,(6)$ and $-1 \,(8)$, 
  and the conjugacy class has dimension $8$;
\item if $\beta^2=-1$, we find the eigenvalues
  \[
    \begin{array}{rr}
      -1 \, (2)
      & -\beta \, (4) \\
      1 \, (4)
      & \beta \, (4)
    \end{array}
  \]
  and the conjugacy class has dimension $10$.
\end{mycpctenum}
In the other cases, the collapsings don't go over
the latter partition $[4^3,2]$, hence
\eqref{OK} holds for all $m$ as we have pointed out above.

\smallskip
B) $\beta^2=1$, that is $\beta = -1$, which gives the eigenvalues
\[
  \begin{array}{rrrrr}
    & \alpha^{-1} \, (2)
    & 1 \, (2)
    & \alpha \, (2)
    & 
    \\
    -\alpha^{-2} \, (1)
    & -\alpha^{-1} \, (2)
    & -1\, (2)
    & -\alpha \, (2)
    & -\alpha^2 \, (1)
  \end{array}.
\]
Further collapsings occur if $\alpha$ is a $2,3,4$-th root of $1$
or a $2,3$-rd root of $-1$. The case $\alpha=-1$ has already been
considered. 
\begin{mycpctenum}[i)]
\item if $\alpha^2=-1$, we find the eigenvalues
  \[
  \begin{array}{rrrrr}
    &
    & 1 \, (4)
    & \alpha \, (4)
    & 
    \\
    &
    & -1\, (2)
    & -\alpha \, (4)
    &
  \end{array}.
  \]
  and the conjugacy class has dimension $10$.
\end{mycpctenum}
In the other cases, the collapsings don't go over the above partition
$[4^3,2]$, hence
\eqref{OK} always holds.

\smallskip
C) $\beta^3=1$, that is $\beta=j$, a primitive cubic root of $1$. We
get the eigenvalues
\[
  \begin{array}{rrrrr}
    \alpha^{-2} \, (1)
    & \alpha^{-1} \, (2)
    & 1 \, (2)
    & \alpha \, (2)
    & \alpha^2 \, (1)
    \\
    & j \alpha^{-1} \, (1)
    & j \, (1)
    & j\alpha \, (1)
    \\
    & j^2 \alpha^{-1} \, (1)
    & j^2 \, (1)
    & j^2 \alpha \, (1)
  \end{array}
\]
The condition giving the largest number of collapsings is
$\alpha^3=1$, in which case we obtain ($\alpha \neq 1$)
\[
  \begin{array}{lll}
    1 \, (4)
    & j \, (5)
    & j^2 \, (5)
  \end{array}.
\]
So we have at most $[5^2,4]$ and \eqref{OK} always holds,
as we have pointed out above.

\smallskip
D) $\alpha = \beta$.
We find the eigenvalues
\[
  \begin{array}{lllllllllll}
    \alpha^{-5}
    & \alpha^{-4}
    & \alpha^{-3}
    & \alpha^{-2}
    & \alpha^{-1}
    & 1
    & \alpha
    & \alpha^{2}
    & \alpha^{3}
    & \alpha^{4}
    & \alpha^{5}
    \\
    (1) & (1) & (1) & (1) 
    & (2) & (2) & (2) & 
    (1) & (1) & (1) & (1) 
  \end{array}      
\]
Collapsings occur for $\alpha$ a root of $1$ of order $o\leq 10$. The
cases $o=2,3$ have been considered already, and for $o \geq 4$  we get
at most the partition $[4^3,2]$ (for $o=4$), hence \eqref{OK}
holds in all cases.

\subsubsection{Conclusion} We conclude from this analysis that 
no non trivial semisimple element of $G_2$ can stabilize a generic
subspace of $\fg_2$ of any dimension from $2$ to $12$. Since this was also the case for unipotent elements, the 
stabilizer of a generic subspace must be trivial in this range of dimensions.

\section{Genus $9$}
\setcounter{genus}{9}

% The case $g=9$ is also relatively straightforward, notably because the critical codimension is only two in this case. 

\subsection{Unipotent elements}
\label{g9:unip}
Nilpotent orbits in $\fsp_{2n}$ are parametrized by partitions $\pi=(\pi_1,\ldots , \pi_k)$ of $2n$ whose odd parts have even multiplicities. 
As is usual we denote by $\pi^*$ the dual partition, and by 
$n_-(\pi)$ the number of odd parts. The 
codimension of the corresponding orbit $\mathcal{O}_\pi$ is given by 
$$\mathrm{codim} (\mathcal{O}_\pi) = \frac{1}{2}\Big( \sum_i(\pi_i^*)^2 +n_-(\pi)\Big).$$

According to \cite[Recipe 5.2.2]{cm}, one obtains a representative of the corresponding orbit as 
follows. Observe that $\fsl_r$ embeds in $\fsp_{2r}$, so counting odd parts with half their multiplicities, we get embeddings 
$$(\prod_{\pi_i\,even}\fsp_{\pi_i})\times (\prod_{\pi_j\,odd}\fsl_{\pi_j}) \subset \fsp_{2n}.$$
Adding regular nilpotent elements of each factor, we get a representative $X_\pi$ of the
nilpotent orbit $\mathcal{O}_\pi$. 

In order to make concrete computations in $\fsp_6$, we first choose a 
basis $e_1,e_2,e_3,e_{-3},e_{-2},e_{-1}$  of $\CC^6$ in which the 
invariant skew-symmetric form writes $\omega = e_1^*\wedge e_{-1}^*+ e_2^*\wedge e_{-2}^*+ e_3^*\wedge e_{-3}^*$. 
The $14$-dimensional module 
$$\wedge^{\langle 3 \rangle}\CC^6 =
\ker(\wedge^3\CC^6\stackrel{\omega}{\rightarrow} \CC^6)$$
admits the basis consisting of the eight vectors 
$f_{\pm\pm\pm}=e_{\pm 1}\wedge e_{\pm 2}\wedge e_{\pm 3}$, 
and the six vectors $g_{\pm k}$, where $k=1,2,3$, given by 
$g_1 = e_1\wedge (e_2\wedge e_{-2}-e_3\wedge e_{-3})$, and so on. 
Our goal is to determine the Jordan type of the action on 
$\wedge^{\langle 3 \rangle}\CC^6$ of a member $X_\pi$ of
each of the eight nilpotent orbits $\mathcal{O}_\pi$ of $\fsp_6$.
Following the above mentioned rule, we provide below an explicit representative  $X_\pi$, in matrix form in our fixed basis, of 
$\mathcal{O}_\pi$. The induced action in our preferred basis of 
$\wedge^{\langle 3 \rangle}\CC^6$ is then easily computed, and 
in particular its Jordan type is readily obtained. 
\begin{alignat*}{2}
X_{[6]}&=\begin{pmatrix}
0 & 1& 0& 0& 0& 0\\  0& 0& 1&  0& 0& 0\\
 0& 0& 0& 1& 0& 0 \\ 0& 0& 0& 0& -1&  0\\  0& 0& 0& 0& 0& -1\\  0& 0& 0&  0& 0& 0
  \end{pmatrix},
& \qquad 
 X_{[42]}&=\begin{pmatrix}
0 & 1& 0& 0& 0& 0\\  0& 0& 0&  0& 1& 0\\
 0& 0& 0& 1& 0& 0 \\ 0& 0& 0& 0& 0&  0\\  0& 0& 0& 0& 0& -1 \\  0& 0& 0&  0& 0& 0
\end{pmatrix},
\\[2mm]
X_{[41^2]} &=\begin{pmatrix}
0 & 1& 0& 0& 0& 0\\  0& 0& 0&  0& 0& 0\\
 0& 0& 0& 0& 0& 0 \\ 0& 0& 0& 0& 0&  0\\  0& 0& 0& 0& 0& -1\\  0& 0& 0&  0& 0& 0
\end{pmatrix},
&
 X_{[3^2]}&=\begin{pmatrix}
0 & 1& 0& 0& 0& 0\\  0& 0& 1&  0& 0& 0\\
 0& 0& 0& 0& 0& 0 \\ 0& 0& 0& 0& -1&  0\\  0& 0& 0& 0& 0& -1 \\  0& 0& 0&  0& 0& 0
\end{pmatrix},
\\[2mm]
X_{[2^3]}&=\begin{pmatrix}
0 & 0& 0& 0& 0& 1\\  0& 0& 0&  0& 1& 0\\
 0& 0& 0& 1& 0& 0 \\ 0& 0& 0& 0& 0&  0\\  0& 0& 0& 0& 0& 0\\  0& 0& 0&  0& 0& 0
  \end{pmatrix},
&
X_{[2^21^2]}&=\begin{pmatrix}
0 & 0& 0& 0& 0& 0\\  0& 0& 0&  0& 1& 0\\
 0& 0& 0& 1& 0& 0 \\ 0& 0& 0& 0& 0&  0\\  0& 0& 0& 0& 0& 0\\  0& 0& 0&  0& 0& 0 
\end{pmatrix},
\\[2mm]
X_{[21^4]}&=\begin{pmatrix}
0 & 0& 0& 0& 0& 0\\  0& 0& 0&  0& 0& 0\\
 0& 0& 0& 1& 0& 0 \\ 0& 0& 0& 0& 0&  0\\  0& 0& 0& 0& 0& 0\\  0& 0& 0&  0& 0& 0
  \end{pmatrix},
&
X_{[1^6]}&=\begin{pmatrix}
0 & 0& 0& 0& 0& 0\\  0& 0& 0&  0& 0& 0\\
 0& 0& 0& 0& 0& 0 \\ 0& 0& 0& 0& 0&  0\\  0& 0& 0& 0& 0& 0\\  0& 0& 0&  0& 0& 0 
 \end{pmatrix}.
\end{alignat*}
 
We finally obtain the following Jordan types:
 $$\begin{array}{llll}
\mathrm{Orbit} & \mathrm{Dimension}  & \mathrm{Jordan\; type} \\   
\mathcal{O}_{[6]} & 18 & 10,4 \\
\mathcal{O}_{[42]} & 16   & 6,4^2 \\
\mathcal{O}_{[41^2]} & 14& 4,3^2,2^2 \\
\mathcal{O}_{[3^2]} &  14 & 5,3^2,1^3 \\
\mathcal{O}_{[2^3]} &  12 & 4,2^5 \\
\mathcal{O}_{[2^21^2]} & 10 & 3^2,2^2,1^4 \\
\mathcal{O}_{[21^4]} & 6 & 2^5,1^4 \\
\mathcal{O}_{[1^6]} & 0 & 1^{14} 
\end{array}$$

\subsection{Semisimple  elements}
\label{g9:smspl}
If $g$ is a semisimple element of $\Sp_6$, with eigenvalues
$t_1,t_2,t_3$ and their inverses,
% and corresponding eigenvectors
% $e_1,e_2,e_3,e_{-1},e_{-2},e_{-3}$,
then the eigenvalues of the induced action on
$\wedge^{\langle 3  \rangle}\CC^6$
are the same six eigenvalues (type I), 
plus the eight products $t_1^{\pm 1}t_2^{\pm 1}t_3^{\pm 1}$
(type II).
% : the eigenvectors for type I are the
% $e_i \wedge e_{i'}\wedge e_{-i'}$ (identified with
% $e_i \wedge e_{i''}\wedge e_{-i''}$ in the quotient $\wedge^{\langle 3
%   \rangle}\CC^6$ of $\wedge^6 \CC^6$, for $\{i,i',i''\}
% = \{1,2,3\}$),
% and those for type II are the $e_{\pm1}\wedge e_{\pm2}\wedge e_{\pm3}$.
% \[
%   \begin{array}{rlllrllll}
%     \text{I:}
%     & t_1,&t_2,&t_3
%     & \text{II}: 
%     & t_1t_2t_3, &t_1t_2^{-1}t_3, &t_1t_2t_3^{-1}, &t_1t_2^{-1}t_3^{-1}
%     \\
%     & t_1^{-1},&t_2^{-1},&t_3^{-1}
%     && t_1^{-1}t_2t_3, &t_1^{-1}t_2^{-1}t_3, &t_1^{-1}t_2t_3^{-1},
%                &t_1^{-1}t_2^{-1}t_3^{-1}
%   \end{array}
% \]
The action of the Weyl group is generated by the permutations of
$\{1,2,3\}$ and all possible exchanges between $t_i$ and its inverse. 

The roots are $t_i^{\pm 1}t_j ^{\pm 1}$ for all $1\leq i < j \leq
3$ and $t_i ^{\pm 2}$ for $1 \leq i \leq 3$, in the multiplicative
notation. 
This gives the following degeneration stratification:
\begin{mycpctenum}[A)]
% 1 d{\'e}g{\'e}n{\'e}rescence
\item $t_2 = t_3$;
\item $t_3^2=1$;
% 2 d{\'e}g{\'e}n{\'e}rescences
\item $t_1=t_2=t_3$;
\item $t_2 = t_3$ and $t_1^2=1$;
\item $t_2 = t_3$ and $t_3^2=1$;
\item $t_2^2=t_3^2=1;$
% 3 d{\'e}g{\'e}n{\'e}rescences
\item $t_1=t_2$ and $t_1^2=t_3^2=1$;
  
\end{mycpctenum}
(The case $t_1=t_2=t_3$ and $t_1^2=1$ is trivial, because then $g =
\pm 1$ belongs to the centre of $\Sp_6$).

\subsubsection{Regular case}
In this case we assume that no root takes the value $1$ on $g$, so
that
$t_1,t_2,t_3,t_1^{-1},t_2^{-1},t_3^{-1}$ are pairwise distinct,
and not equal to $\pm1$.
Then the conjugacy class of $g$ has dimension $18$.

Collapsings among type II are of the following
kinds:
\begin{mycpctenum}[a)]
\item $t_1t_2t_3 = t_1^{-1}t_2^{-1}t_3$, if
  $t_1^2t_2^2=1$;
\item $t_1t_2t_3 = t_1^{-1}t_2^{-1}t_3^{-1}$, provided
  $t_1^2t_2^2t_3^2=1$. 
\end{mycpctenum}
It follows that type II eigenvalues can at most collapse in pairs
in this regularity stratum.
Since on the other hand
type I eigenvalues are pairwise distinct, in the present case all
eigenspaces in the representation have dimension at most $3$.
Then \eqref{OK} holds in all cases:
as we have already done before, we use our python toolkit to write
down all partitions of $14$ as sums of integers not larger than $3$,
then select among them those maximal for the partial order of
Remark~\ref{rk:monotone}, and compute for each of those the maximum of
Proposition~\ref{dim-stable}~(1) for all $m=2,\ldots,12$.

\subsubsection{$t_2=t_3$}
In this case we have the eigenvalues
\[
  t_1^{\pm1}\, (3)
  \quad t_2^{\pm 1} \, (2)
  \quad t_1^{\pm1}t_2^{\pm2} \, (1).
\]
We assume that $t_1^{\pm1},t_2^{\pm 1}$ are $4$ pairwise distinct
values not equal to $\pm1$
(the case when they coincide is a further degeneration
class), hence the conjugacy class of $g$ has dimension $16$.

By our assumption, collapsings must involve the eigenvalues
$t_1^{\pm1}t_2^{\pm2}$, and can be of the following kinds
(as always, up to the Weyl action): a) $t_1^2=t_2^2$; b) $t_2^3=t_1$;
c) $t_2^4=1$; d) $t_1^2=t_2^4$. \smallskip
%\begin{mycpctenum}[a)]
%\item $t_1^2=t_2^2$;
%\item $t_2^3=t_1$;
%\item $t_2^4=1$;
%\item $t_1^2=t_2^4$.
%\end{mycpctenum}

% Assume $t_1^2=t_2^2$. Then $t_2^3=t_1$ implies $t_2^4=1$, and in this
% case $t_1=t_2^{-1}$ which is excluded in this degeneration stratum.
% If $t_1^2=t_2^2$ then $t_2^2=1$ which is excluded as well.

Assume $t_1^2=t_2^2$, hence $t_2=-t_1$. Then the eigenvalues are
\[
  t_1^{\pm1}\, (4)
  \quad -t_1^{\pm 1} \, (2)
  \quad t_1^{\pm3}\, (1).
\]
The only further collapsings in this degeneration class are if
$t_1^4=-1$ or $t_1^6=1$ (note that if $t_1=i$, then
$t_2=t_1^{-1}$),
and they give the partitions
$[4^2,3^2]$ and
$[4^2,2^3]$ respectively. \eqref{OK} holds in all cases: as usual by
now, it suffices by monotonicity to check the two latter cases, which
we do with our python toolkit.

Assume $t_1=t_2^3$. Then the eigenvalues are
\[
  t_2^{\pm3}\, (3)
  \quad t_2^{\pm 1} \, (3)
  \quad t_2^{\pm5} \, (1).
\]
It is excluded in this case that $t_2$ be a root of $1$ of order
$2,3,4,6$ so the possible collapsings happen when it is a root of
order $8$ or $10$. They give the partitions
$[4^2,3^2]$ and $[3^4,2]$, and \eqref{OK} holds in all cases
as follows from the checkings already carried out.

Assume $t_2=i$. Then we have the eigenvalues
\[
  t_1^{\pm1}\, (3)
  \quad \pm i \, (2)
  \quad -t_1^{\pm1} \, (2),
\]
and no further collapsing is possible. We thus have the partition
$[3^2,2^4]$, and \eqref{OK} holds in all cases.

Eventually, if  $t_1^2=t_2^4$
we get the eigenvalues
\[
  t_1^{\pm1}\, (3)
  \quad t_2^{\pm 1} \, (2)
  \quad t_1t_2^{-2} \, (2)
  \quad (t_1t_2^{2})^{\pm1} \, (1)
\]
and no new further collapsing is possible,
so that again \eqref{OK} holds in all cases.

\subsubsection{$t_3^2=1$}
In this case we have the eigenvalues
\[
  t_1^{\pm1},t_2^{\pm1}\, (1)
  \quad t_3 \, (2)
  \quad t_3t_1^{\pm1}t_2^{\pm1} \, (2)
\]
and the conjugacy class of $g$ has dimension $16$.

Type I eigenvalues are pairwise distinct.
Up to the Weyl action,
the only possible collapsing between types I and II is
$t_1^2 = t_3t_2$, 
and the only possible one in type  II is
$t_1^2 = t_2^2$.

Assume $t_1^2 = t_3t_2$. Then also
$t_1^{-2} = t_3t_2^{-1}$, and the eigenvalues are
\[
  t_1^{\pm1}\, (3), \quad
  t_2^{\pm1}\, (1)
  \quad t_3 \, (2)
  \quad t_3t_1t_2,\,
  t_3t_1^{-1}t_2^{-1}\, (2).
\]
A further collapsing happens only if
either $t_2=t_3t_1^{-1}t_2^{-1}$ ($\iff t_1^5=t_3$),
or $(t_1t_2)^2=1$ ($\iff t_1^6=1$).
They give the partitions
$[3^4,2]$ and $[4,3^2,2,1^2]$ respectively,
and \eqref{OK} holds in all cases.

Assume $t_1^2 = t_2^2$. Possible combinations with collapsings between
types I and II have been considered above, so the only new possibility
of a further collapsing is that simultaneously $t_1^2=t_2^{-2}$, which
is possible only if $t_2=\pm i$ and $t_1=-t_2$, which takes us to
another stratum.

\subsubsection{$t_1=t_2=t_3$}
We call $t$ the common value, which in this stratum is assumed to verify
$t^2 \neq 1$. In this case the conjugacy class of $g$ has dimension
$12$ because of the $6$ new relations $t_i/t_j=1$, $i\neq j$,
and the eigenvalues are
\[
  \begin{array}{lll}
    t^{\pm1}\, (6)
    & t^{\pm3}\, (1).
  \end{array}
\]
Collapsings occur if either $t^4=1$ or $t^6=1$, in which cases we get
the partitions $[7^2]$ and $[6^2,2]$ respectively.
In the latter case \eqref{OK} holds for all $m$, and for all $m>2$ in
the former.

With $t=\pm i$,
we obtain for $m=2$ (resp.\ $m=12$)
a $12$-dimensional family of pairs $(L,g)$ with $g.L=L$ by
considerings sums of two lines in the two eigenspaces
(resp.\ the dual configuration).
\gag
\label{gag9.1}

\subsubsection{$t_2 = t_3$ and $t_1^2=1$}
In this case we have the eigenvalues
\[
  \begin{array}{lll}
    t_1\, (6)
    & t_2^{\pm1} \, (2)
    & t_1 t_2^{\pm2} \, (2)
  \end{array}
\]
with $t_2 \neq \pm1$, and
the conjugacy class of $g$ has dimension $14$.

The possible collapsings are $t_2=t_1t_2^{-2}$
($\iff t_2^3=t_1$), and $t_2^4=1$, which cannot happen simultaneously.
They give the partitions
$[6,4^2]$ and $[6,4,2^2]$
respectively, and \eqref{OK} holds in all cases.

\subsubsection{$t_2 = t_3$ and $t_3^2=1$}
In this case we have the eigenvalues
\[
  \begin{array}{lll}
    t_1^{\pm1}\, (5)
    & t_2 \, (4)
  \end{array}
\]
with $t_1 \neq \pm1$, and
the conjugacy class of $g$ has dimension $12$
($6$ relations $t_2/t_3, t_2^2, t_3^2$ and their inverses).
No collapsing is possible, and we have the only partition
$[5^2,4]$, for which \eqref{OK} holds for all $m$.

\subsubsection{$t_2^2=t_3^2=1$}
We may assume that $t_2=1$, $t_3=-1$, and $t_1 \neq \pm1$.
We have the eigenvalues
\[
  \begin{array}{lll}
    t_1^{\pm1}\, (1)
    & \pm1 \, (2)
    & - t_1^{\pm1}\, (4)
  \end{array}
\]
and
the conjugacy class of $g$ has dimension $14$.

The only possible collapsing happens when $t_1^2=-1$, in which case we
have the partition $[5^2,2^2]$, and \eqref{OK} holds in all cases.

\subsubsection{$t_1=t_2$ and $t_1^2=t_3^2=1$}
In this case we may assume that $t_1=t_2=-1$ and $t_3=1$.
We have the four relations $t_1/t_2,t_1t_2,t_1^2,t_2^2,t_3^2$ and their
inverses, so the conjugacy class of $g$ has dimension $8$.
(For verification: all roots take the value $1$ except
$t_1/t_3,t_2/t_3,t_1t_3, t_2t_3$ and their 
inverses, so the conjugacy class of $g$ has dimension $8$ indeed).

The eigenvalues in the representation are
\[
  \begin{array}{lll}
    -1\, (4)
    & 1 \, (10).
  \end{array}
\]
\eqref{OK} holds for $m>2$, but we find a $16$-dimensional family of
$2$-planes fixed by $g$ by considering those $2$-planes inside the
$10$-dimensional eigenspace.
\\
\gag
\label{gag9.2}

\subsubsection{Conclusion}
We have found that if $P$ is a generic $k$-plane with
$2<k<12$, then its stabilizer is trivial, whereas if $P$ is generic of
dimension $2$ or $12$, then its stabilizer may contain only elements as
described in the two cases \ref{gag9.1} and \ref{gag9.2}, and only finitely many of them.

\subsection{Codimension two}
\label{s:g9.codim2}

Let $P\subset \wedge^{\langle 3\rangle}\CC^6\subset \wedge^3\CC^6$ be
a generic plane.
It follows from the analysis above
% in \S \ref{g9:unip} and \ref{g9:smspl} above
that the stabilizer of $P$ is finite,
and its non-trivial elements may only be of the two following kinds
(up to sign):
\begin{mycpctenum}[I)]
\item  involutions $\Id_A-\Id_{A^\perp}$ with $A$ a
  non-degenerate plane,
  provided $P \subset A \otimes \wedge ^{\langle 2 \rangle} A^\perp$
  (case \ref{gag9.2});
  following our usual notation,
  $\wedge ^{\langle 2 \rangle} A^\perp$ is the kernel of the
  contraction by $\omega$ on $\wedge^2 A^\perp$;
\item anti-involutions $i(\Id_E-\Id_F)$ with $E,F$ transverse
  Lagrangian subspaces, provided
  $P$ is the sum of two lines in $\wedge^3E\oplus (E\otimes\wedge^2F)$
  and $(\wedge^2E\otimes F)\oplus \wedge^3F$ respectively (case
  \ref{gag9.1}); 
  beware that the two latter spaces are not entirely contained in
  $\wedge^{\langle 3\rangle}\CC^6$.
\end{mycpctenum}
Proposition~\ref{reduction} below tells us that there are indeed
elements of type I, and that it is always possible to decompose
$A^\perp$ as the sum of two orthogonal non-degenerate planes $A_2,A_3$
in such a way that
$P \subset A \otimes A_2 \otimes A_3 \subset 
A \otimes \wedge ^{\langle 2 \rangle} A^\perp$.
Elements of type II are taken care of in
Proposition~\ref{reduction2}, and the stabilizer of $P$ is completely
described in Corollary~\ref{coro:stab9}.

\begin{prop} \label{reduction}
Let $P\subset \wedge^{\langle 3\rangle}\CC^6\subset \wedge^3\CC^6$ be a generic plane. 
\begin{enumerate}
\item There exists a unique triple $(A_1,A_2,A_3)$ of non-degenerate, pairwise orthogonal planes 
in $\CC^6$, such that $P\subset A_1\otimes A_2\otimes A_3\subset  \wedge^{\langle 3\rangle}\CC^6$.
\item
  The three planes $A_1,A_2,A_3$ are the only
  non-degenerate $2$-planes $A \subset \CC^6$ such that
  $P \subset A \otimes \wedge ^{\langle 2 \rangle} A^\perp$.
\end{enumerate}
\end{prop}

\begin{proof} The proof partly relies on the fact that the Lagrangian Grassmannian $\LG(3,6)\subset \PP( \wedge^{\langle 3\rangle}\CC^6)$ 
is a variety with {\it one apparent double point}. In other words, given a general point $x$ in $\PP( \wedge^{\langle 3\rangle}\CC^6)$,
there exists a unique bisecant to $\LG(3,6)$ passing through $x$ \cite[Example 2.9]{cr}. 

Let us apply this observation to $x=[p]$, for $p$ a general point of $P$. This means that we can write $p$ in the form
$p=u_1\wedge u_2\wedge u_3+ v_1\wedge v_2\wedge v_3$, where $U=\langle u_1, u_2, u_3\rangle $ and $V=\langle v_1, v_2, v_3\rangle $
are Lagrangian subspaces of $\CC^6$, in general position, and uniquely defined by $p$. Now consider another general point $p'$ of $P$,
again with its two associated  Lagrangian subspaces $U'$, $V'$ of $\CC^6$. Under the generality assumption, we can describe $U'$ and 
$V'$ as the graphs of two isomorphisms $\alpha$ and $\beta$ from $U$ to $V$. Moreover $\alpha\circ\beta^{-1}$ is in general semisimple; 
let $f_1, f_2, f_3$ be a basis of eigenvectors in $V$, with distinct eigenvalues $\lambda_1, \lambda_2, \lambda_3$. Let $e_i= \beta^{-1}(f_i)$.
Since $u_1\wedge u_2\wedge u_3$ is a multiple of $e_1\wedge e_2\wedge e_3$ and $v_1\wedge v_2\wedge v_3$ is a multiple of $f_1\wedge f_2\wedge f_3$,
we conclude that there exist scalars $a, b, a', b'$ such that 
$$\begin{array}{lcl}
p &=&ae_1\wedge e_2\wedge e_3+bf_1\wedge f_2\wedge f_3, \\
p'&=& a'(e_1+f_1)\wedge (e_2+f_2)\wedge (e_3+f_3)+ \\
 & & \hspace*{3cm} +b'(e_1+\lambda_1f_1)\wedge (e_2+\lambda_2f_2)\wedge (e_3+\lambda_3f_3).
\end{array}$$
Letting $A_i=\langle e_i,f_i\rangle$, we deduce that $P\subset A_1\otimes A_2\otimes A_3$. Note that the isotropy of $U, V, U', V'$ implies 
that $A_1, A_2, A_3$ are pairwise orthogonal. 

Now observe that the pairwise orthogonal triples $(A_1,A_2,A_3)$ are parametrized by a variety $X$ of dimension $12$, so that the 
relative Grassmannian $Y$ of planes in $A_1\otimes A_2\otimes A_3$ has dimension $12+12=24$. This is also the dimension of 
$\G(2, \wedge^{\langle 3\rangle}\CC^6)$, and since the previous observations imply that the natural map $\pi: Y\rightarrow \G(2, \wedge^{\langle 3\rangle}\CC^6)$ is dominant, it must be generically finite. As a consequence, the triple $(A_1, A_2, A_3)$ that we have constructed from two general points 
$p, p'$ of $P$ cannot change when we vary $p$ and $p'$, so it must be canonically defined by $P$. We conclude that $\pi$ is in fact birational, 
which implies assertion (1).

\medskip
Let us turn to assertion (2). We first make some points for
future use in the proof. Let
$A \subset \CC^6$ a non-degenerate $2$-plane.
Observe that
$\PP(A\otimes \wedge^{\langle 2\rangle}A^\perp)$ contains the 
Segre product $\PP(A)\times \QQ^3(A^\perp)$, where $\QQ^3(A^\perp)$ is
the hyperplane section of $\G(2,A^\perp)$ defined by $\omega$,
that is the intersection $\G(2,A^\perp) \cap \PP(\wedge^{\langle
  2\rangle}A^\perp)$ inside $\wedge^{2}A^\perp$.
Moreover, this Segre product is contained in $\LG(3,6)$. 
We shall use the fact that the Segre product
$\PP^1\times\QQ^3 \subset \PP^9$ is a variety with one apparent double
point as well (see e.g. \cite[Example 2.4]{cr}).

Let $C_P \subset \LG(3,6)$ be the curve described by the points $x,x'
\in \LG(3,6)$ such that $[p] \in \langle x,x' \rangle$ when $[p]$
varies in $\PP(P)$.
Note that $P \subset A\otimes \wedge^{\langle 2\rangle}A^\perp$ if and
only if $C_P \subset \PP(A\otimes \wedge^{\langle 2\rangle}A^\perp)$:
the if part follows from the fact that $\PP(P)$ is contained in the
span of $C_P$, and the only if part from the fact that
$\PP(A)\times \QQ^3(A^\perp)$ is a variety with one apparent double
point inside its span, so that for general $[p] \in \PP(P)$ the two
points $x,x' \in \LG(3,6)$ necessarily lie on $\PP(A)\times
\QQ^3(A^\perp)$.

Note that a point $x \in \LG(3,6)$ lies in
$\PP(A\otimes \wedge^{\langle 2\rangle}A^\perp)$
if and only if the corresponding Lagrangian
$3$-plane $\Pi_x \subset \CC^6$ intersects $A$ non-trivially:
the only if part is tautological, and conversely
if there exists a non-zero $a \in \Pi_x \cap A$, then
$\Pi_x \subset a^\perp$, hence
the intersection of $\Pi_x \cap A^\perp$ is a plane,
and therefore $x$ belongs to $\PP(A)\times \QQ^3(A^\perp)$.
Moreover, when this holds the intersection $\Pi_x \cap A$ is
necessarily a line, since $A$ is non-degenerate and $\Pi_x$ is
isotropic.

% We claim that $C_P$ is contained in
% $\PP(A\otimes \wedge^{\langle 2\rangle}A^\perp)$
% if and only if for all $x\in C_P$, the corresponding Lagrangian
% $3$-plane $\Pi_x \subset \CC^6$ intersects $A$ non-trivially.
% To prove this, let $x \in C_P$.
% If there exists a non-zero $a \in \Pi_x \cap A$, then
% $\Pi_x \subset a^\perp$ and therefore
% the intersection of $\Pi_x \cap A^\perp$ is a plane,
% so that $x$ belongs to $\PP(A)\times \QQ^3(A^\perp)$.

% In fact, let us show that for all $x\in C_P$,
% $x \in \PP(A\otimes \wedge^{\langle 2\rangle}A^\perp)$
% if and only if $\Pi_x \cap A \neq (0)$.

\medskip
Now let us eventually prove assertion (2).
Let $A \subset \CC^6$ a non-degenerate $2$-plane
such that
$P \subset A\otimes \wedge^{\langle 2\rangle}A^\perp$.
Let $x$ be a general point on $C_P$. We have seen in the proof of
assertion (1) that $C_P\subset \PP(A_1)\times\PP(A_2)\times \PP(A_3)$,
so the Lagrangian $3$-plane $\Pi_x$ is the direct sum of three lines
$L_1(x)\subset A_1$,  $L_2(x)\subset A_2$,  $L_3(x)\subset A_3$.
Our assumption that
$P \subset A\otimes \wedge^{\langle 2\rangle}A^\perp$ implies that
$\Pi_x$ and $A$ intersect along a line $L_A(x)$.
% generated by some
% vector $a_1(x)+a_2(x)+a_3(x)$, $a_i(x) \in L_i(x)$ for $i=1,2,3$.

Consider the family of automorphisms
$\sigma = s_1 \Id_{A_1}+ s_2 \Id_{A_2}+ s_3 \Id_{A_3}$,
$s_i \in \CC^*$ for $i=1,2,3$. For all such $\sigma$, the plane
$\sigma(A)$ intersects $\Pi_x$ along the line
$\sigma(L_A(x))$ for all $x\in C_P$
(indeed $\sigma$ leaves the lines $L_i(x)$ fixed, hence also
$\Pi_x=L_1(x)+L_2(x)+L_3(x)$), and
therefore $C_P$ is contained in
$\PP(\sigma(A)\otimes \wedge^{\langle 2\rangle}\sigma(A)^\perp)$.
Considering general such $\sigma$'s, we thus obtain a family of
non-degenerate planes $\sigma(A)$ such that
$P \subset \sigma(A)\otimes \wedge^{\langle
  2\rangle}\sigma(A)^\perp$.
Such planes being only finitely many by our analysis in
\S \ref{g9:smspl}, we must have $\sigma(A)=A$ for all $\sigma$.
This implies that $A$ is the sum of two lines in $A_i$ and $A_j$
respectively. Since $A_i$ and $A_j$ are orthogonal whereas $A$ is
non-degenerate, we must have $i=j$ and assertion (2) is proved.
\end{proof}

\begin{prop} \label{reduction2}
Let $P\subset \wedge^{\langle 3\rangle}\CC^6\subset \wedge^3\CC^6$ be
a generic plane, and let $A_1,A_2,A_3$ be as in
Proposition~\ref{reduction}. 
Then there exist exactly twelve pairs $(E,F)$ of transverse
Lagrangian subspaces of $\CC^6$,
such that $E$ and $F$ both meet all three $A_i$'s non trivially,
and $P$ meets non trivially
$\wedge^3E\oplus (E\otimes\wedge^2F)$ and
$(\wedge^2E\otimes F)\oplus \wedge^3F$.
\end{prop}

\begin{proof}
Suppose that we have decomposed each
$A_i$ into the direct sum of two lines,
$A_i=E_i\oplus F_i$. There is an
induced decomposition
$A_1\otimes A_2\otimes A_3=A_E \oplus A_F$, with\\
\resizebox{\linewidth}{!}{%
\parbox{\textwidth}{%
\begin{align*}
  A_E &= (E_1\otimes E_2\otimes E_3) \oplus
           (E_1\otimes F_2\otimes F_3)
           \oplus (F_1\otimes E_2\otimes F_3)
           \oplus (F_1\otimes F_2\otimes E_3), \\
    A_F &= (F_1\otimes F_2\otimes F_3)
           \oplus (F_1\otimes E_2\otimes E_3)
           \oplus (E_1\otimes F_2\otimes E_3)
           \oplus (E_1\otimes E_2\otimes F_3).
\end{align*}
}}
There are $6$ parameters for the six lines $E_i,F_i$, and then $3+3$ parameters for choosing a line in $A_E$ and a line in $A_F$; taking their
direct sum, this gives a family of planes in $A_1\otimes A_2\otimes A_3$ with twelve parameters; since $12$ is also the dimension of
$\G(2,A_1\otimes A_2\otimes A_3)$, we can expect that a generic plane $P$ can be obtained in this way. In order to check that this guess
is correct, we compute the generic rank of the differential of the following map $\eta$. Let $Q_i$ denote the complement of the diagonal
in $\PP (A_i)\times \PP(A_i)$. Over $Q= Q_1\times Q_2\times Q_3$, there are two rank four vector bundles $\mathcal{A}_E$ and  $\mathcal{A}_F$
defined by the formulas above; they are both sub-bundles of the trivial vector bundle with fiber $A_1\otimes A_2\otimes A_3$, and 
the direct sum map induces the morphism
$$\eta : Z=\PP(\mathcal{A}_E)\times_Q \PP(\mathcal{A}_F) \longrightarrow \G(2,A_1\otimes A_2\otimes A_3)$$ 
that we claim is dominant.
To check this, we fix a basis $\alpha_i, \alpha_{-i}$ of $A_i$. Local coordinates on an open subset of $Q$ are 
obtained by considering in $A_i$ the lines $E_i$ and $F_i$ generated
by $e_i=\alpha_i+x_i\alpha_{-i}$ and $f_{i}=\alpha_{-i}+y_i\alpha_i$.
We then get local relative coordinates on $Z$ by considering lines
generated by
$d=e_1e_2e_3+p_1e_1f_2f_3+p_2f_1e_2f_3
+p_3f_1f_2e_3$ and
$d'=f_1f_2f_3+q_1f_1e_2e_{3}+q_2e_1f_2e_{3}+q_3e_1e_{2}f_3$
(for brevity we omit the tensor product signs).  
At first order, 
we compute that 
$$\begin{array}{lcl}
d & = & \alpha_1\alpha_2\alpha_3+x_1\alpha_{-1}\alpha_2\alpha_3+x_2\alpha_1\alpha_{-2}\alpha_3+x_3\alpha_1\alpha_2\alpha_{-3}+ \\
 & & \hspace*{3cm}+p_1\alpha_1\alpha_{-2}\alpha_{-3}+p_2\alpha_{-1}\alpha_2\alpha_{-3}+p_3\alpha_{-1}\alpha_{-2}\alpha_3, \\
d' & = & \alpha_{-1}\alpha_{-2}\alpha_{-3}+y_1\alpha_1\alpha_{-2}\alpha_{-3}+y_2\alpha_{-1}\alpha_2\alpha_{-3}+y_3\alpha_{-1}\alpha_{-2}\alpha_3+\\
 & & \hspace*{3cm}+q_1\alpha_{-1}\alpha_2\alpha_{3}+q_2\alpha_1\alpha_{-2}\alpha_{3}+q_3\alpha_1\alpha_{2}\alpha_{-3},
\end{array}$$
which implies our claim. 
Note that $E=E_1\oplus E_2\oplus E_3$ and $F=F_1\oplus F_2\oplus F_3$ are Lagrangian subspaces of $\CC^6$ and that the symplectic 
automorphism $s=i(\Id_E-\Id_F)$ leaves $P$ invariant.

\medskip
Now suppose that there is another decomposition $A_i=E'_i\oplus F'_i$
compatible with $P$, hence two other Lagrangian subspaces $E'$ and
$F'$ of $\CC^6$ such that the symplectic automorphism
$t=i(\Id_{E'}-\Id_{F'})$ also leaves $P$ invariant.
Then also $u=st$ leaves $P$ invariant, hence we must have 
$u^2 = \pm \Id$ by the analysis of \S \ref{g9:unip} and
\ref{g9:smspl}. 

If $u^2=\Id$, then $st=ts$ and therefore, the decompositions
$A_i=E_i\oplus F_i=E'_i\oplus F'_i$, which are given by the
eigenspaces of $s$ and $t$, must be the same. This only leaves the
possibility to exchange $E_i$ with $F_i$.
The eight possible permutations give four pairs
of Lagrangian spaces $(E,F)$.

If $u^2=-\Id$, then $st=-ts$, so that for all $j=1,2,3$ the action of
$t$ on $A_j$ exchanges the eigenspaces of $s$. We can thus find
generators $e_j$ of $E_j$ and $f_j$ of $F_j$ such that $t(e_j)=f_j$
and $t(f_j)=-e_j$. The eigenspaces of $t$ and $u$ acting on $A_j$
are then the lines generated by $e_j \pm im_j$ and $e_j \pm
m_j$, respectively, so that the associated pairs of Lagrangian spaces
are
$$\begin{array}{ll} 
E'=\langle e_1 + if_1, e_2 + if_2, e_3 + if_3\rangle, &  F'=\langle e_1 - if_1, e_2 - if_2, e_3 - if_3\rangle, \\
E''=\langle e_1 + f_1, e_2 + f_2, e_3 + f_3\rangle, &  F''=\langle e_1 - f_1, e_2 - f_2, e_3 - f_3\rangle .
\end{array}$$
Moreover, in this case the plane $P$ is generated by 
two vectors of the form
$$\begin{array}{lcl} 
p & = & xe_1\wedge e_2\wedge e_3+y_1e_1\wedge f_2\wedge f_3+y_2f_1\wedge e_2\wedge f_3+y_3f_1\wedge f_2\wedge e_3, \\
p' & = & xf_1\wedge f_2\wedge f_3+y_1f_1\wedge e_2\wedge e_3+y_2e_1\wedge f_2\wedge e_3+y_3e_1\wedge e_2\wedge f_3.
\end{array}$$
Indeed, $P$ contains a vector as $p$ above, and since it is stable
under $t$ it also contains $t(p)=p'$, which is linearly independent
from $p$. Similarly, $u(p)=ip'$.

Eventually, note that there can be no other decomposition of the  $A_i$'s compatible with $P$, since $t$ and $u$ are the only anti-involutions
up to sign, that exchange the eigenspaces of $s$.

\medskip
Conversely, we claim that the pairs $(E',F')$ and $(E'',F'')$ as
above, and correspondingly $t$ and $u$ indeed exist.  To
see this, recall that we can always generate $P$, up to sign, by two
vectors
$$\begin{array}{lcl} 
p & = & xe_1\wedge e_2\wedge e_3+y_1e_1\wedge f_2\wedge f_3+y_2f_1\wedge e_2\wedge f_3+y_3f_1\wedge f_2\wedge e_3, \\
p' & = & x'f_1\wedge f_2\wedge f_3+y'_1f_1\wedge e_2\wedge e_3+y'_2e_1\wedge f_2\wedge e_3+y'_3e_1\wedge e_2\wedge f_3.
\end{array}$$
Multiplying $f_i$ by some scalar $\tau_i$, we can always reduce to the
case where $x'=x$ and $y'_i=y_i$ for each $i$, if the coefficients
$x,x',y_i,y'_i$ are non zero, a condition which holds by genericity of
$P$. 
Then $P$ is preserved by $t$ and $u$, and the rest follows.
\end{proof}

\begin{coro}\label{coro:stab9}
The stabilizer in $\PSp_6$ of a generic plane $P\subset \wedge^{\langle 3\rangle}\CC^6$ is isomorphic to $(\ZZ/2\ZZ)^4$. 
\end{coro}

\proof We have established in \S \ref{g9:unip} and
\ref{g9:smspl} that the elements of  $\Sp_6$ that stabilize such a
generic plane $P$ are involutions 
of two possible types I and II as described at the beginning of \S
\ref{s:g9.codim2}.
If $(A_1,A_2,A_3)$ is the unique triple of non degenerate, pairwise orthogonal planes in $\CC^6$ such that $P\subset A_1\otimes A_2\otimes A_3$, the stabilizer contains the involutions 
$\pm \Id_{A_1}\pm \Id_{A_2}\pm \Id_{A_3}$. These generate in $\PSp_6$ a
copy of $(\ZZ/2\ZZ)^2$.
Moreover, with the previous notations, the three involutions defined
in $\PSp_6$ by $s$, $t$, $u=st$, generate another copy of $(\ZZ/2\ZZ)^2$.
Note that $s$ and $t$ are only defined up to sign, this ambiguity
being absorbed by the first three involutions.
All these elements in $\PSp_6$ commute one with another, although
their representatives in $\Sp_6$ may anticommute.
We thus get a copy of $(\ZZ/2\ZZ)^4$ inside the stabilizer of $P$.

Let us prove that this $(\ZZ/2\ZZ)^4$ is indeed the whole stabilizer
of $P$. Consider an element $r$ of the latter.  
If $r$ is of type I, i.e., an involution $\pm(\Id_A-\Id_{A^\perp})$
for some non degenerate plane $A$ of $\CC^6$, then by
Proposition~\ref{reduction} $r$ is one of
$\pm \Id_{A_1}\pm \Id_{A_2}\pm \Id_{A_3}$ and belongs to our
$(\ZZ/2\ZZ)^4$.
So suppose that $r$ is of type II, i.e., an anti-involution
$i(\Id_E-\Id_F)$ associated to a pair $(E,F)$ of tranverse  
Lagrangian subspaces of $\CC^6$.
Let $a \neq \pm \Id$ be one of the involutions $\pm\Id_{A_1}\pm
\Id_{A_2}\pm \Id_{A_3}$ that stabilize $P$. 
Consider $u=ra$. We know that $u^2=\pm \Id$, so $ra=ar$ or
$ra=-ar$. In the latter case, $r$ permutes the eigenspaces of $a$,
which are two of dimensions $2$ and $4$, a contradiction.
So $a$ and $r$ commute, and therefore $E,F$ are as in
Proposition~\ref{reduction2}. 
\qed

\medskip
Note that the first copy of $(\ZZ/2\ZZ)^2$ considered in this proof acts trivially on $\PP(P)$ since 
$P\subset A_1\otimes A_2\otimes A_3$, so $\Aut_G(P)\simeq
(\ZZ/2\ZZ)^2$ is generated by the three involutions 
$s,t,u$ only. 

On the dual side, we get $\Aut_G(P^\perp)=(\ZZ/2\ZZ)^4$, and this is
the automorphism group of the general codimension two linear section
$X$ of $\LG(3,6)$. This completes the proof of our main Theorem in
genus $9$. 
We have shown more precisely that $\Aut(X)$ consists of two different
types of involutions: three of type I (coming  
from elements of $\Sp_6$ with eigenspaces of dimensions $2$ and $4$),
and $12$ of type II (with two eigenspaces of dimensions $3$).

\begin{prop}
The fixed locus in $X$ of an involution of type I is a Del Pezzo
surface of degree four (not in its anticanonical embedding). The fixed
locus of an involution of type II is the disjoint union of two
Veronese surfaces.   
\end{prop}

\proof Suppose that $X$ is defined by $P^\perp$, for $P$ a generic plane of $\wedge^{\langle 3\rangle}\CC^6$. If $s$ is an 
involution of type I, there exists a non isotropic plane $A\subset \CC^6$ such that $P\subset A\otimes\wedge^{\langle 2\rangle}A^\perp$, and 
$s=\Id_A-\Id_{A^\perp}$. We have seen that the eigenspace decomposition of the induced action is 
$$\wedge^{\langle 3\rangle}\CC^6=A\otimes\wedge^{\langle 2\rangle}A^\perp \oplus K,$$
where $K\subset \wedge^2A\otimes A^\perp\oplus \wedge^3A^\perp$ is  the kernel of the contraction map to $A^\perp$ by $\omega$. 
We deduce the eigenspace decomposition of $P^\perp$ as $P_+\oplus P_{-}$, where $P_+$ is the orthogonal to $P$ in 
$A\otimes\wedge^{\langle 2\rangle}A^\perp$ and $P_{-}\simeq K^*$. The fixed locus of $s$ is then the union of the intersections 
of $\LG(3,6)$ with $\PP (P_+)$ and $\PP(P_{-})$. It is easy to check that the latter is empty: the intersection of $\PP(\wedge^2A\otimes A^\perp\oplus \wedge^3A^\perp)$ with $\G(3,6)$ consists in the three-planes that either contain $A$, or are contained in $A^\perp$, and none of those 
is isotropic. The former is a general condimension two linear section of the intersection of $\PP(A\otimes\wedge^{\langle 2\rangle}A^\perp)$
with $\G(3,6)$, which consists in those three planes that are generated by a line of $A$ and an isotropic plane in $A^\perp$. 
Geometrically, this is a codimension two linear section of $\PP^1\times \QQ^3$, hence a Del Pezzo surface of degree four. 

\smallskip 
Now consider the case where $s$ is an involution of type II; then there exist two transverse Lagrangian subspaces $E$ and $F$ 
of $\CC^6$ such that $s=i(\Id_E-\Id_F)$. The induced eigenspace decomposition of  $\wedge^{\langle 3\rangle}\CC^6$ is into 
seven-dimensional isotropic spaces 
$$U_+=\wedge^3E\oplus K_+, \qquad \mathrm{and} \qquad  U_-=\wedge^3F\oplus K_-,$$
where $K_+\subset E\otimes \wedge^2F$ (resp. $K_-\subset F\otimes \wedge^2E$) is the kernel of the contraction map to $F$ (resp. $E$) 
by $\omega$. Moreover, $P$ is generated by two generic lines in $U_+$ and $U_-$, and dually, $P^\perp$ is the direct sum of a general 
hyperplane in $U_+$ with  a general hyperplane in $U_-$. Note that $E$ and $F$ are in duality through $\omega$, and that once we have 
identified $F$ with $E^*$, we get $K_+$ as the image of the natural (Koszul type) map from $S^2E\otimes \wedge^3 E^*$ to $E\otimes \wedge^2 E^*$. 
It is then easy to check that the intersection of $\PP (U_+)\simeq \PP(\CC\oplus S^2E)$ with $\LG(3,6)$ is a cone over a Veronese 
surface. Cutting with a general hyperplane we get the Veronese surface back, and our second claim follows. \qed 

\medskip 
The intersection of $\PP(U_+)$ with $\PP(A_1)\times \PP(A_2)\times \PP(A_3)$ consists of the four points 
$\langle L_1, L_2, L_3\rangle$ plus $\langle L_1, M_2, M_3\rangle $ and its permutations. In general these 
points do not belong to the curve $C_P$, which must be preserved by any automorphism stabilizing $P$. 
Note that $C_P$ is by definition a double cover of $\PP(P)$, the branch locus being the intersection of $\PP(P)$ 
with the tangent variety to $\LG(3,6)$. This tangent variety being a quartic hypersurface (see, e.g., \cite[Proposition 6.4]{lm}), $C_P$ is in fact an elliptic curve. We conclude that the type II involutions must restrict on $C_P$ to translations by two-torsion 
points. This is very similar to what we observed in genus $8$.

\section{Genus $7$}
\setcounter{genus}{7}

Now comes the hardest case. A first mild difficulty is that it involves a spin module for a Lie group of rank
five, with $16$ unipotent orbits and an important number of cases to consider for semisimple elements. From our 
analysis of unipotent and semisimple elements, we will conclude that in the critical codimension four we need to consider 
certain special involutions, related to splitting of spin modules under restriction to smaller Lie groups. 
These splittings are well known as subrepresentations, but we will need to be very specific about the corresponding 
subspaces of the spin module in order to be able to understand the generic automorphism groups. 

\subsection{A brief reminder on spin modules}
\label{s:spin}
Suppose that $V=V_{2n}$ is a complex vector space of dimension $2n$, endowed with a non degenerate 
quadratic form $Q$. Chose a decomposition $V=E\oplus F$ into maximal isotropic subspaces. So
$Q$ vanishes on $E$ and $F$, and defines a perfect duality between $E$ and $F$. 

Let $\Delta:=\wedge^*E$ denote the exterior algebra of $E$. It admits a natural action of $E$ 
defined by the wedge product, and also a natural action by contraction with $Q$. Explicitely, 
if $x=a+\alpha$ with $a\in E$ and $\alpha\in F$, we have for example in degree one 
$$x.e=a\wedge e+Q(\alpha, e)1.$$
If $x'=a'+\alpha'$ is another vector, its action on $x.e$ is given by 
$$x'.(x.e)=a'\wedge a\wedge e+Q(\alpha', a)e-Q(\alpha', e)a+Q(\alpha, e)a'$$
since $\alpha'.1=0$. This is only partially skew-symmetric in $x$ an $x'$, and 
$$x'.(x.e)+x.(x'.e)=\Big( Q(\alpha', a)+Q(\alpha, a')\Big) e = Q(x,x')e.$$
This formula extends if we replace $e$ by any element in $\Delta$. Thus it extends to 
an algebra action of the Clifford algebra $Cl(V,Q)$ on $\Delta$, and then of the spin group $\Spin(V,Q)$ 
by restriction (see \cite{chevalley}). In the sequel we will only need the infinitesimal action of the Lie 
algebra of the spin group, which is simply $\wedge^2V$. This action is readily deduced
from the previous formulas, through the identifications
$$x\wedge x'=\frac{1}{2}(x.x'-x'.x)=x.x'-\frac{1}{2}Q(x,x')1$$
as operators on $\Delta$. Since the action of such an operator on $\Delta$ preserves its 
$\ZZ/2\ZZ$-grading, the spin module $\Delta$ actually splits into the direct sum of the 
two half-spin representations 
$$\Delta_+=\wedge^+E, \qquad \Delta_-=\wedge^-E,$$
which are both irreducible, of the same dimension $2^{n-1}$. 

Note that this construction is not canonical, in the sense that it relies on the initial choice 
of the decomposition $V=E\oplus F$, while the half-spin representations don't (at least up to isomorphism; they are, in 
particular, indistinguishable). This is a major source of
complications, as we will see in the sequel.

In the rest of this section, we shall write $\Delta$ for either one of
the half-spin representations $\Delta_\pm$ of $\Spin_{10}$, and
reserve the notation $\Delta_\pm$ for certain subpieces of our
half-spin representation of $\Spin_{10}$. These subpieces will be defined by
restricting the representation to a copy of
$\Spin_4\times\Spin_6$. Suppose for instance that 
$\CC^{10}$ has been split into the direct sum $U\oplus U^\perp$ for some non degenerate four-plane $U$, 
and that 
$E=E'\oplus E''$ has been chosen to be the sum of an isotropic plane $E'\subset U$ and 
an isotropic three-plane $E''\subset U^\perp$. Then the formulas 
\begin{equation}\label{decspin}
\wedge^+E=\wedge^+E'\otimes \wedge^+E''\oplus \wedge^-E'\otimes \wedge^-E'', 
\end{equation}
\begin{equation}
\wedge^-E=\wedge^+E'\otimes \wedge^-E''\oplus \wedge^-E'\otimes \wedge^+E''
\end{equation}
show that the spin representations indeed split into pieces of the same dimension.
This is actually a general fact, independent of the possibility to
split $E$ compatibly with $U$, as we shall see later on.

\subsection{Unipotent elements}
\label{s:g10.nilpotent}
There are $16$ unipotent orbits in $\Spin_{10}$, corresponding to the $16$ nilpotent orbits 
in $\mathfrak{so}_{10}$. Classically, they are indexed by partitions of $10$ in which even parts
have even multiplicities. Recall that if $\pi$ is such a partition, $\pi^*$ is the dual partition, and 
$n_-(\pi)$ denotes the number of odd parts. The 
codimension of the corresponding orbit $\mathcal{O}_\pi$ is given by 
$$\mathrm{codim} (\mathcal{O}_\pi) = \frac{1}{2}\Big( \sum_i(\pi_i^*)^2 -n_-(\pi)\Big).$$
We deduce that the dimensions are the following:
$$\begin{array}{lll}
Partition \qquad  & Dimension  \qquad 
& Jordan\; type\\
9,1  & 40 & 5,11\\ 7,3 & 38 & 2,6,8 \\ 5,5 & 36 & 1,3,5,7 %(1^2,3^3,5)
\\ 7,1^3 & 36 & 1^2,7^2 \\ 
5,3,1^2 & 34 & 3^2,5^2 %(2^4,4^2)
\\ 5,2^2,1 & 32 & 3,4^2,5 %(1^2,2^2,3^2,4)
\\ 4^2,1^2 & 32 & 1^3,4^2,5 \\
3^3,1 & 30 & 2^4,4^2 \\ 5,1^5 & 28& 4^4 %(1^4,3^4) 
\\ 3^2,2^2 & 28 &1^2,2^2,3^2,4 \\ 3^2,1^4 & 26& 1^4,3^4 \\ 3,2^2,1^3 & 24& 1^2,2^4,3^2\\ 2^4,1^2 & 20 & 1^5,2^4,3\\ 3,1^7 & 16 & 2^8\\ 2^2,1^6 & 14& 1^8,2^4
\\ 1^{10} & 0& 1^{16}
\end{array}$$
A representative $X_\pi$ of each nilpotent orbit $\mathcal{O}_\pi$ can be obtained as follows
(see \cite[Recipe 5.2.6]{cm}).
Denote by $(\mu_1,\ldots, \mu_k)$ the odd parts of $\pi$. Denote its even parts by $(\nu_1,\ldots, \nu_k)$,
counted with half their multiplicities (always even). Consider the natural embeddings
$$\mathfrak{so}_{\mu_1}\times\cdots\times\mathfrak{so}_{\mu_k}\times 
\mathfrak{sl}_{\nu_1}\times\cdots\times\mathfrak{sl}_{\nu_l}\subset 
\mathfrak{so}_{\mu_1}\times\cdots\times\mathfrak{so}_{\mu_k}\times 
\mathfrak{so}_{2\nu_1}\times\cdots\times\mathfrak{so}_{2\nu_l}\subset \mathfrak{so}_{10}.$$
Choose regular nilpotent elements $X_i$ in $\mathfrak{so}_{\mu_i}$, $Y_j$ in $\mathfrak{sl}_{\nu_j}$.
Then one can let
$$X_\pi =X_1+\cdots + X_k+Y_1+\cdots +Y_l\in \mathcal{O}_\pi.$$
The restriction of a half-spin representation of $\mathfrak{so}_{10}$ to 
$\mathfrak{so}_{\mu_1}\times\cdots\times\mathfrak{so}_{\mu_k}\times 
\mathfrak{so}_{2\nu_1}\times\cdots\times\mathfrak{so}_{2\nu_l}$ will split into a sum
of tensor products of spin representations for the different factors; this decomposition can easily be obtained inductively, since a half-spin representation of $\mathfrak{so}_{2m}$ will restrict to 
$\Delta_a\otimes \Delta_b$ on $\mathfrak{so}_{2a+1}\times \mathfrak{so}_{2b+1}$ for $m=a+b+1$, and 
to $\Delta_a^+\otimes \Delta_b^+\oplus \Delta_a^-\otimes \Delta_b^-$ on $\mathfrak{so}_{2a}\times \mathfrak{so}_{2b}$ for $m=a+b$;  while a spin representation of $\mathfrak{so}_{2m+1}$ will restrict 
to  $\Delta_a^+\otimes \Delta_b\oplus \Delta_a^-\otimes \Delta_b$ on $\mathfrak{so}_{2a}\times \mathfrak{so}_{2b+1}$ for $m=a+b$. Moreover, restricting a half-spin representation of 
$\mathfrak{so}_{2a}$ to $\mathfrak{sl}_{a}$ yields the even (or odd) exterior algebra of the 
natural representation. 

In order to compute the Jordan type of the action of each $X_\pi$ on a half-spin representation of 
$\mathfrak{so}_{10}$, it is therefore enough to know the Jordan type of the action on a spin
representation of a regular nilpotent element of $\mathfrak{so}_{2m+1}$ for $m\le 4$, and the Jordan
type of the action of a regular nilpotent element of $\mathfrak{sl}_{n}$ on the exterior algebra
of the natural representation, for $n\le 4$. These are given as follows: 
$$\begin{array}{rccrcc}
\mathfrak{so}_{3} & 2 & &\mathfrak{sl}_{2} & 1^2 & 2 \\
\mathfrak{so}_{5} & 4&&\mathfrak{sl}_{3}&1,3& 1,3\\
\mathfrak{so}_{7} &1,7&&\mathfrak{sl}_{4}&1^3,5&4^2 \\
\mathfrak{so}_{9} &5,11&&&& 
\end{array}$$
One finally deduces the Jordan type of  $X_\pi$ on a half-spin representation by using 
Proposition \ref{Jordan-tensor}.  The results are given in the table above. 

Arguing as indicated in \S \ref{s:strategy-nilp}, we conclude that for
$m=4,\ldots,12$, the general $m$-dimensional linear subspace $L \subset
\Delta_+$ has no unipotent element in its stabilizer.

%\medskip\noindent {\it Remark}. As already noticed in genus $8$, 
%in order to exclude the possibility that the generic stabilizer may contain a unipotent element, 
%it would be enough to exhibit one subspace whose
%stabilizer is made of semisimple elements only. 
%In dimension $8$ for example we could let 
%$$U_8=\langle 1, e_{14}, e_{23}, e_{45}, e_{35},  e_{1234}, e_{1235}, e_{1245}\rangle ,$$
%where we chose a basis $e_1,\ldots , e_5$ of some isotropic subspace $E$ of $\CC^{10}$ and
%identified the half-spin representation $\Delta_+$ with even part of the exterior algebra of $E$. 
%Unfortunately we have not been able to apply this more direct approach to all the cases
%we are interested in. 

\subsection{Semisimple elements}
Now let $g$ be a semisimple element in $\SO_{10}$, that we may suppose to belong to a 
standard maximal torus. If $\epsilon_1, \ldots , \epsilon_5$ are the diagonal characters 
of this torus, the characters of the half-spin representation are the 
$$\frac{1}{2}(\pm \epsilon_1 \pm \cdots \pm \epsilon_5),$$
with an even number of plus signs. In other words, let
$t_1=s_1^2,\ldots, t_5=s_5^2$ and their inverses 
denote the eigenvalues 
of $g$, with $V_t\subset\CC^{10}$ the eigenspace corresponding 
to the eigenvalue $t$. Then the eigenvalues of the induced action on $\PP(\Delta)$ must be (up to some scalar)
the $s_1^{\pm 1}\cdots s_5^{\pm 1}$, again with an even number of plus signs; up to some scalar, 
these are the same as $1$, the ten products $t_it_j$ (that we call eigenvalues of type I), and the five products $t_pt_qt_rt_s$ (type II).  
%This implies that the collapsings that can happen must be of the form $t_1^{a_1}\cdots t_5^{a_5}=1$,
%with $a_i\in\{-1,0,1\}$, and an odd number of zeroes. Up to symmetries, we have two types of such
%collapsings, represented by $t_1t_2=1$ and $t_1t_2t_3t_4=1$; we will call them collapsings of type 
%I and II, respectively. 

% \[
% \begin{array}{lll}
%   1 \ (1)
%   & \text{I: }  t_it_j \ (1)
%   & \text{II: }  t_it_jt_kt_l \ (1)
% \end{array}
% \]

The roots are $t_i^{\pm 1}t_j ^{\pm 1}$ for all $1\leq i < j \leq
5$.
The Weyl action is generated by the permutations of the $t_i$'s, and
the exchanges of two $t_i$'s with their inverses.
In particular, a relation of the form
$t_it_j=1$ can always be replaced by $t_i=t_j$ since we have the 
freedom of exchanging $t_j$ with $t_j^{-1}$; such a relation will
therefore take us to another type.

\subsubsection{Regular case}
This means that we consider those semisimple $g\in \SO_{10}$ such that
the $t_i$'s are pairwise distinct.
Then the conjugacy class of $g$ has dimension $40$.

We may assume that no eigenvalue of type I equals $1$, because a
relation $t_1t_2=1$ may be replaced by $t_1=t_2$ by acting with the
Weyl group, and this takes us to another regularity class.
The ten eigenvalues of type I can at worst collapse in pairs, because
of the regularity assumption on $g$. Type II eigenvalues are pairwise
distinct. It is possible to pair a type I eigenvalue with one of type
II, but then these two eigenvalues cannot be paired with anything
else.
It is also possible to pair a type II eigenvalue with $1$, but up to
the Weyl action this is equivalent to a collapsing between two type I
eigenvalues. 

The upshot is that the eigenvalues for the action on $\Delta$ can
collapse at most in pairs, and we can make at most $7$ pairs (by
pairing each type II with one type I, and making two pairs of pure
type I). One checks that \eqref{OK} holds in all cases:
by monotonicity it suffices to check the case of the eigenspace
decomposition $[2^7,1^2]$, which is done in our python verification
package.

\subsubsection{Subregular case}
This means that only two eigenvalues coincide; 
suppose this is $t_4=t_5$. For the action on $\Delta$, the eigenvalues
are 
\[
\begin{array}{rrr}
  1 \ (1) & t_1t_2, t_1t_3, t_2t_3 \ (1) & t_1t_2t_3t_4 \ (2)
  \\
          & t_1t_4, t_2t_4, t_3t_4 \ (2) & t_1t_2t_4^2 \ (1)
  \\
          & t_4^2 \ (1) & t_1t_3t_4^2 \ (1)
  \\
          && t_2t_3t_4^2 \ (1)
\end{array}
\]
For generic values of the $t_i$'s, the conjugacy class of $g$ has
dimension $38$.

Among type I, $t_1t_2, t_1t_3, t_2t_3$ are pairwise distinct, and so
are $t_1t_4, t_2t_4, t_3t_4, t_4^2$.
The type II eigenvalues are pairwise distinct, 
as always in a given degeneration stratum.
Moreover $t_1t_2t_3t_4$ may be paired only with $1$ and $t_4^2$.
It follows that the eigenspaces have dimension at most
$4$ at any rate.
It follows that \eqref{OK} holds in all
cases, except if we have a partition containing either $[4^2,3^2]$ or
$[4^3]$:
to prove this, we (i) enumerate all possible partitions of $16$ with
pieces of size not greater than $4$,
(ii) remove from the list all those that contain either $[4^2,3^2]$ or 
$[4^3]$,
(iii) select maximal partitions (with respect to the partial order
relation described in Remark~\ref{rk:monotone}),
and (iv) apply Proposition~\ref{dim-stable} (1) to all
of them; this is carried out with our python script.

% the simplest way to verify this is arguably to brutally

% and to apply Proposition~\ref{dim-stable} (1) to all
% of them; we carry this out with our python script, but note that the
% corresponding output is left apart from the rest in the output file,
% as it is so lengthy.

It is not possible that
$t_1t_2, t_1t_3, t_2t_3$ collapse all at the same time with
$t_3t_4, t_2t_4, t_1t_4$ respectively  without
violating our subregularity assumption,
but it is possible that
$t_1t_2, t_1t_3, t_1t_2t_3t_4$ collapse with
$t_3t_4, t_2t_4, t_4^2$, iff
$t_2,t_3,t_4$ equal $-1,1,-t_1$ respectively (up to exchanging $t_2$
and $t_3$).
This may conveniently be checked using Macaulay2, see the ancillary
files listed in \S \ref{s:anciliary}. 
In the latter case, the eigenvalues are (writing $t$ for $t_1$)
\[
\begin{array}{rrr}
  \pm 1 \ (1) & \pm t, \pm t^2 \ (3) & \pm t^3 \ (1)
\end{array}
\]
(note that we automatically have the fourth collapsing
$t_1t_4=t_2t_3t_4^2$).
The special values $t=\pm i$ are forbidden by the subregularity
assumption (otherwise $t_1t_4=1$), so we obtain at most the partition
$[3^4,2^2]$, and it is not possible to end up with a partition
containing either $[4^2,3^2]$ or $[4^3]$, hence \eqref{OK} always
holds by what we have said above.

The other possibilities amount to those investigated above
by taking the Weyl action into account.
For instance, if $t_1t_4, t_2t_4$ collapse with
$t_2t_3t_4^2, t_1t_3t_4^2$, we have $t_1=t_2t_3t_4$
and $t_2=t_1t_3t_4$, and those relations may be changed into
$t_1t_4=t_2t_3$ and $t_2t_4=t_1t_3$ by exchanging $t_4,t_5$ with their
inverses.

\subsubsection{$t_3=t_4=t_5$}
In this case we assume that no root other than
$t_3/t_4,t_3/t_5,t_4/t_5$ and their inverses take the value $1$, so
the conjugacy class of $g$ has dimension $34$.
The eigenvalues of the action of $g$ on $\Delta$
are
\[
\begin{array}{rrr}
  1 \ (1) & t_1t_2 \ (1) & t_1t_2t_3^2 \ (3)
  \\
        & t_1t_3 \ (3) & t_1t_3^3 \ (1)
  \\
        & t_2t_3 \ (3) & t_2t_3^3 \ (1)
  \\
        & t_3^2 \ (3)
\end{array}
\]

The eigenvalues with multiplicity $3$ may collapse with at most one
other, the latter necessarily with multiplicity $1$.
Thus we cannot get any eigenspace of dimension larger than $4$, and
this is enough for \eqref{OK} to hold in all cases
(we verify this in the same way as in the subregular case, using our 
python script).

\subsubsection{$t_2=t_3$ and $t_4=t_5$}
%We have $4$ roots and no other taking the value $1$, hence
The conjugacy class of $g$ has dimension at most $36$, and may be
strictly smaller if $t_2^2$ or $t_4^2$ take the value $1$.
The eigenvalues of the action on
$\Delta$ are
\[
\begin{array}{rrr}
  1 \ (1) & t_1t_2 \ (2) & t_1t_2^2t_4 \ (2)
  \\
        & t_1t_4 \ (2) & t_1t_2t_4^2 \ (2)
  \\
        & t_2t_4 \ (4) & t_2^2t_4^2 \ (1)
  \\
        & t_2^2 \ (1)
  \\
  & t_4^2 \ (1)
\end{array}
\]
We make the following observations:
\begin{mycpctenum}[i)]
\item $t_1t_2$ may collapse only with
  a) $t_4^2$, b) $t_1t_2t_4^2$, c) $t_2^2t_4^2$;
  b) happens iff $t_4^2=1$ and excludes a) and c);
  a) and c) may happen at the same time, iff $t_2^2=1$.
  A strictly analogous observation holds for all the eigenvalues with
  multiplicity $2$.
\item $t_2t_4$ may not collapse with any other eigenvalue.
\end{mycpctenum}
It follows that the eigenspaces in $\Delta$ have at most dimension
$4$. Then \eqref{OK} holds for all cases, except for
$m=4$ if the partition is $[4^4]$ and the conjugacy class has
dimension $36$, as one verifies as in the previous cases (again this
is included in our python verifications).
But if the partition is $[4^4]$, then necessarily either
$t_2^2$ or $t_4^2$ equals $1$, hence the conjugacy class of $g$ has
dimension at most $34$ so that \eqref{OK} holds in this case as well,
as has been verified in the previous case.

\subsubsection{$t_1=t_2$ and $t_3=t_4=t_5$}
For generic values of $t_1$ and $t_3$, exactly $8$ roots take the
value $1$, hence the dimension of the conjugacy class is at most
$32$.
The eigenvalues of the action on $\Delta$ are
\[
\begin{array}{rrr}
  1 \ (1) & t_1^2 \ (1) & t_1^2t_3^2 \ (3)
  \\
        & t_1t_3 \ (6) & t_1t_3^3 \ (2)
  \\
        & t_3^2 \ (3) 
\end{array}
\]
We observe that:
\begin{mycpctenum}[i)]
\item $t_1t_3$ may only collapse with $t_1t_3^3$, iff
  $t_3^2=1$;
\item $t_3^2$ may only collapse with $1,t_1^2,t_1^2t_3^2$, and a
  similar statement holds for $t_1^2t_3^2$.
\end{mycpctenum}

Let us first consider the main degeneration stratum for this case,
i.e., we assume that neither $t_1^2$ nor $t_3^2$ equals $1$. Then
$t_1t_3$ cannot collapse, $t_3^2,t_1^2t_3^2$ may only collapse with
$t_1^2,1$ respectively, and the two latter are mutually exclusive, and
$t_1t_3^3$ may only collapse with $1$ or $t_1^2$. So at most we have
the partition $[6,4,3^2]$, and \eqref{OK} holds in all cases
(we verify this case with our python script, and the other follow
by monotonicity).

Now assume that $t_1^2=1$. Then the conjugacy class of $g$ has
dimension at most $30$ (we have the two extra relations
$t_1t_2=t_1^{-1}t_2^{-1}=1$). In this case the eigenvalues become
\[
\begin{array}{rrrr}
  1 \ (2) & t_1t_3 \ (6) & t_3^2 \ (6) & t_1t_3^3 \ (2)
\end{array}
\]
If $t_3^2\neq 1$, the only possible further collapsing is
$1=t_1t_3^3$, which gives the partition $[6^2,4]$, and then \eqref{OK}
holds in all cases, as we verify following our usual method.
If $t_3^2= 1$, we obtain the partition $[8^2]$ and the conjugacy class
has dimension $24$ ($6$ new relations $t_3t_4,t_3t_5,t_3t_4$ and their
inverses).
In this case we find a $48$-dimensional family of pairs $(L,g)$ with
$g.L=L$ by considering those $L$ that are the sums of two $2$-planes
in each of the two $8$-dimensional eigenspaces.
Note that in this case, letting $U$ be the $4$-plane sum of the
four eigenlines of respective weights $t_1,t_2,t_1^{-1},t_2^{-1}$
(which all take the same value $\pm 1$ on $g$), one
has $g = \pm(\Id_U-\Id_{U^\perp})$.
\gag
\label{gag7}

Eventually, let us check the case when $t_3^2=1$ but $t_1^2\neq 1$.
Then the conjugacy class has dimension $26$ and the eigenvalues are
\[
\begin{array}{rrr}
  1 \ (4) & t_1t_3 \ (8) & t_1^2 \ (4).
\end{array}
\]
There is no possible further collapsing, and
\eqref{OK} holds in all cases.

\subsubsection{$t_2=t_3=t_4=t_5$}
The eigenvalues on $\Delta$ are
\[
\begin{array}[c]{rrrr}
  1 \ (1) & t_1t_2 \ (4) & t_1t_2^3 \ (4)
  \\
        & t_2^2 \ (6)&  t_2^4 \ (1) 
\end{array}
\] 
Let us first consider the case $t_2^2\neq 1$.
Then the only possible collapsings are $t_1t_2=t_2^4$, iff
$t_1=t_2^3$, and $t_1t_2^3=1$, iff $t_1^{-1}=t_2^3$. So at most we
have the partition $[6,5^2]$, and the conjugacy class of $g$ has
dimension $28$, so that \eqref{OK} holds in all cases.

If $t_2^2=1$ the eigenvalues become
\[
\begin{array}[c]{rrrr}
  1 \ (8) & t_1t_2 \ (8)
\end{array}
\]
and there is no further possible collapsing. In this case
the conjugacy class has dimension $16$, and \eqref{OK} holds in all
cases.

\subsubsection{$t_1=t_2=t_3=t_4=t_5$}
The eigenvalues on $\Delta$ are
\[
\begin{array}[c]{rrrr}
  1 \ (1) & t^2 \ (10) & t^4 \ (5)
\end{array}
\]
If $t^2=1$ the action of $g$ on $\Delta$ is trivial so we discard this
case. Thus $t^2\neq 1$, the conjugacy class of $g$ has dimension $20$,
and the only possible collapsing is $t^4=1$, which gives the partition
$[10,6]$. \eqref{OK} holds in all cases.

\subsubsection{Conclusion}\label{concl:Spin_10}
The stabilizer of a generic subspace  $P \subset \Delta$ of 
dimension $5$ to $11$ is trivial.
If $P$ has dimension $4$ or $12$, non-trivial elements in its
stabilizer must be of the kind described in case \ref{gag7} above.

\subsection{Codimension four}

By the previous study,  a general $4$-plane $P\subset\Delta$ may only
be stabilized by a finite number of involutions $\pm t_U$, with
$t_U=\Id_{U}-\Id_{U^\perp}$ in $\SO_{10}$ for some 
non-degenerate four-plane $U \subset \CC^{10}$.
The restriction of the half-spin representation $\Delta$
to $\fso(U)\times\fso(U^\perp)$ decomposes into
the direct sum of two eight-dimensional sub-representations $\Delta_+$
and $\Delta_-$
(recall our notation convention at the end of
subsection~\ref{s:spin}),
and $P$ is stabilized by the induced action of $t_U$ 
if and only if it is the direct sum of two $2$-planes
$P_+\subset \Delta_+$ and $P_-\subset \Delta_-$.

We shall prove (see Theorem~\ref{t:g7.concl}) that there exist exactly 
three non degenerate four-planes $U,V,W\subset\CC^{10}$ satisfying the
above conditions, hence $P$ is only stabilized by the three
corresponding involutions $t_U,t_V,t_W$. 
We will see that $U, V, W$ must be in 
very special relative position: their pairwise intersections will be 
non degenerate planes $A,B,C$. 

This will be the conclusion of a  detour, along which 
we will need to understand spin modules and their splittings 
under restrictions to such subalgebras of $\fso_{10}$ as 
$\fso(U)\times\fso(U^\perp)$.
%The whole point of \S \ref{s:split8} to \ref{s:Delta_i} is to obtain
%an \emph{explicit} description of these splittings
%(which is achieved in Theorem~\ref{t:split}), so as to be able
%to compute the rank of a suitable differential which is the keystone
%of the proof of Theorem~\ref{t:g7.concl}.

\subsubsection{How to split a spin module in eight dimensions}
\label{s:split8}
As a warm-up, 
let $V_4\subset\CC^{8}$ be a non degenerate four-dimensional subspace, and $V'_4=V_4^\perp$ its orthogonal 
with respect to the quadratic form $Q$. The restriction of the spin 
representation $\Delta$ to $\fso(V_4)\times \fso(V'_4)$ splits into two four-dimensional submodules. 
How can we identify them
concretely? 

\smallskip
Recall that the half-spin representations can be defined as $\Delta_+=\wedge^+E$ and $\Delta_-=\wedge^-E$, once a splitting 
$\CC^{8}=E\oplus F$ into transverse Lagrangian subspaces has been fixed. This clearly implies that there is a natural 
equivariant map from $\CC^{8}\otimes\Delta_\pm$ to $\Delta_\mp$, defined by wedge products and contractions by vectors 
of $E$ and $F$ in $\CC^8$. Iterating, one obtains natural morphisms from $\wedge^{2p}\CC^{8}\otimes\Delta_\pm$ to $\Delta_\pm$, for any $p$. 

In particular, in $\wedge^{4}\CC^{8}$ we can consider the Pl{\"u}cker line associated to 
$V_4$. This line induces an endomorphism $\psi_{V_4}$ of $\Delta_\pm$, well defined up to scalars. 
Being canonically defined by $V_4$, the eigenspace decomposition of this endomorphism must be compatible with 
the structure of $\Delta_\pm$ as a module over $\fso(V_4)\times \fso(V'_4)$. So if $\psi_{V_4}$ is not a homothety, 
this endomorphism has no other choice than to admit two four dimensional eigenspaces: the two submodules of the 
decomposition. This is indeed what will happen, and this yields an efficient method in order to locate concretely 
these
submodules (the existence of which we know a priori only by abstract arguments).

\smallskip
Now suppose that $V_4$ is transverse to $F$, so that it can be defined as the graph of a morphism $\Gamma\in \Hom(E,F)$. If $V_4$ is 
also transverse to $E$, this morphism is an isomorphism. In this case, a crucial observation is that 
one can define a canonical element of $\wedge^4E$, up to sign, by letting 
$$\gamma =\frac{v_1\wedge v_2\wedge v_3\wedge v_4}{\det
  Q(v_i,\Gamma(v_j))^{1/2}}$$ %_{1\le i,j\le 4}
with $(v_1,v_2,v_3,v_4)$ any basis of $E_4$.

\begin{prop}
\label{p:spin4}
As representations of $\fso(V_4)\times \fso(V'_4)$, the half-spin representations $\Delta_+$ and $\Delta_-$ split 
into $\Delta_+=\delta_+\oplus\delta_-$ and $\Delta_-=\delta'_+\oplus\delta'_-$, where 
$$\delta_{\pm} = \langle 1\pm\gamma, \;
\theta\pm\Gamma(\theta).\gamma, \;\theta\in\wedge^2E\rangle, \qquad  
\delta'_{\pm}=\langle e\pm\Gamma(e).\gamma, \; e\in E\rangle.$$
\end{prop}

\begin{proof} An explicit computation shows that $\psi_{V_4}$ acts, as expected, as homotheties on 
each of these subspaces, with opposite factors. (And by equivariance, this computation needs to 
be done only for one specific $V_4$.)
\end{proof}

%\medskip
Observe that since $\gamma$ is only defined up to sign, the two modules $\delta_+$ and $\delta_-$, as well as  $\delta'_+$ 
and $\delta'_-$, are in fact indistinguishable, as must be the rule for half-spin representations.

\subsubsection{How to split a spin module in ten dimensions}
Let now $V_4\subset\CC^{10}$ be a non degenerate four-dimensional subspace, and $V_6=V_4^\perp$ its orthogonal with respect to the quadratic form $Q$. The restriction of a half-spin 
representation $\Delta$ to $\fso(V_4)\times \fso(V_6)$ splits into two eight-dimensional submodules. We want to identify them concretely. 

Suppose that $\Delta$ has been constructed as $\wedge^+E$, where $\CC^{10}=E\oplus F$ is a fixed splitting into Lagrangian spaces. 
When $V_4$ is transverse to both $E$ and $F$, its two projections are isomorphisms onto subspaces $E_4\subset E$ and $F_4\subset F$, 
and $V_4$ can be defined as the graph of an isomorphism $\Gamma\in \Hom(E_4,F_4)$. Moreover, since $E$ and $F$ are in perfect duality through 
the quadratic form $Q$, the hyperplanes $E_4$ and $F_4$ are orthogonal to lines $F_1\subset F$ and $E_1\subset E$, such that in 
general, $E=E_1\oplus E_4$ and $F=F_1\oplus F_4$. Observe that this yields a splitting 
$$\Delta=\wedge^+E= \wedge^+E_4\oplus E_1\otimes\wedge^-E_4.$$

\begin{prop}\label{8dec}
As a representation of  $\fso(V_4)\times \fso(V_6)$, the half-spin representation $\Delta$  splits
into $\Delta=\delta^8_+\oplus\delta^8_-$, where 
$$\delta^8_+=\delta_+\oplus E_1\otimes \delta'_-, \qquad
\delta^8_-=\delta_-\oplus E_1\otimes \delta'_+$$
with $\delta_\pm^{\pprime}$ defined as in Proposition~\ref{p:spin4}.
\end{prop}

\begin{proof} Exactly as for Proposition~\ref{p:spin4}, this only
  requires the computation of how $\psi_{V_4}$ acts, and only for one
  specific $V_4$. We leave this to the reader.
\end{proof}

\subsubsection{How to split a spin module from a triple of four
  planes}
\label{s:Delta_i}
Now consider the following situation: three orthogonal, non degenerate planes $A, B, C\subset \CC^{10}$ are given, and 
we want to describe a simultaneous splitting with respect to the three four planes $U, V, W$ that are sums of two of those planes. 

\begin{lemma}\label{4dec} There exists a unique decomposition of $\Delta$ into the direct sum of four-dimensional subspaces 
$\Delta_1, \Delta_2, \Delta_3, \Delta_4$ such that the decompositions of $\Delta$ as sums of submodules are given by 
$$\begin{array}{lcl}
\Delta_{|\fso(U)\times \fso(U^\perp)} & = & (\Delta_1\oplus \Delta_2)\oplus (\Delta_3\oplus \Delta_4), \\
\Delta_{|\fso(V)\times \fso(V^\perp)} & = & (\Delta_1\oplus \Delta_3)\oplus (\Delta_2\oplus \Delta_4), \\
\Delta_{|\fso(W)\times \fso(W^\perp)} & = & (\Delta_1\oplus \Delta_4)\oplus (\Delta_2\oplus \Delta_3).
\end{array}$$
\end{lemma}

Note that $\Delta_1, \Delta_2, \Delta_3, \Delta_4$ are well defined only up to permutations by pairs. 

\proof The four dimensional space $U$ defines a line in $\wedge^4\CC^{10}$, hence an endomorphism $\psi_U$ of $\Delta$, up to a scalar. 
This operator has two eigenspaces of dimension eight, corresponding to two opposite eigenvalues, which are nothing else than the 
two components of the restriction of $\Delta$ to $\fso(U)\times \fso(U^\perp)$. 

If $U=A\oplus B$, we also have two associated operators $\psi_A$ and
$\psi_B$ defined by the Pl{\"u}cker lines of $A$ and
$B$ in $\wedge^2 \CC^{10}$.
The orthogonality of $A$ and $B$ implies that  $\psi_U$ 
is proportional to $\psi_A\psi_B=\psi_B\psi_A$.  Moreover $\psi_A$ and $\psi_B$ also have two eigenspaces of dimension eight, corresponding to two 
opposite eigenvalues, which are the two components of the restriction of $\Delta$ to $\fso(A)\times \fso(A^\perp)$ and 
 $\fso(B)\times \fso(B^\perp)$. We can normalize them, up to a sign, so that the two eigenvalues are $\pm 1$, and then normalize
 $\psi_U$ as $\psi_A\psi_B$. 
 
 If $V=A\oplus C$ and $W=B\oplus C$, with the same normalizations we get that $\psi_V=\psi_A\psi_C$ and $\psi_W=\psi_B\psi_C$. 
 But then, since $\psi_C^2=1$ we deduce that $\psi_W=\psi_U\psi_V$. This implies the claim after simultaneous diagonalization
 of $\psi_U$ and $\psi_V$. \qed

\subsubsection{Conclusion of the proof}
We are now set to deduce the following result from the above analysis. 

\begin{theorem}\label{t:g7.concl}
Let $P\subset\Delta$ be a general $4$-dimensional subspace.
There exists a triple $(A,B,C)$ of mutually orthogonal non-degenerate 
planes in $\CC^{10}$ such that $P$ is stabilized by the three
involutions $t_U,t_V,t_W$ associated to the four-planes $U=A\oplus B$, $V=A\oplus
C$, $W=B\oplus C$, and only by those involutions (and their opposites).
\end{theorem}

As a direct corollary we obtain that the stabilizer of $P$ in
$\Spin_{10}$ is isomorphic to $(\ZZ/2\ZZ)^2$,
which proves our main Theorem~\ref{symetries} in genus $7$.
Recall that $t_U=\Id_{U}-\Id_{U^\perp}$ and that all elements in the
stabilizer of $P$ are necessarily of this type, see \ref{concl:Spin_10}.

\begin{proof}
The fact that a triple $(A,B,C)$ as in the statement should exist is
indicated by the following dimension count.
Let $\mathcal{U}\subset \G(2,10)^3$ be the variety of unordered orthogonal
triples $\{A,B,C\}$ of non degenerate planes;
it has dimension $36$ 
($16$ parameters for a non isotropic plane $A$, then $12$ for a non
isotropic plane $B$ orthogonal to $A$, finally $8$ for a non isotropic
plane $C$ orthogonal to both $A$ and $B$).
Each $\{A,B,C\} \in \mathcal U$ decomposes 
the half-spin representation $\Delta$ into four $4$-dimensional
subspaces
$\Delta_1\oplus \Delta_2\oplus \Delta_3\oplus \Delta_4$, and by
generality $P$ is stabilized by the three involutions associated to
$U,V,W$ if and only if
it is the direct 
sums of four lines contained in these four subspaces.
This gives
$4\times 3$ additional parameters,
hence in total $36+12=48$ parameters, which is exactly the
dimension of $\G(4,\Delta_+)$.  

To confirm this dimension count, let us consider the map 
$$\varphi :
\{A,B,C\} \in \mathcal{U}
\longmapsto (\Delta_1,\Delta_2,\Delta_3,\Delta_4) \in \G(4,\Delta_+)^4/\Sn[4]$$
(where we need to mod out by the symmetric group
$\Sn[4]$ since the four components of $\Delta_+$ are not
well-defined individually). The folllowing statement will conclude  the proof of the Theorem. 

\begin{lemma}\label{injectivity}
The map
$\varphi$ is injective. 
\end{lemma}

\noindent {\it Proof of the Lemma}. Suppose that, like at the end of 
section $6.1$, 
$\CC^{10}$ has been split into the direct sum $U\oplus U^\perp$ for some non degenerate four-plane $U$, 
and that 
$E=E'\oplus E''$ has been chosen to be the sum of an isotropic plane $E'\subset U$ and 
an isotropic three-plane $E''\subset U^\perp$. Then the spin module splits accordingly as in 
equation (\ref{decspin}), that we rewrite as 
$$\Delta_+=\Delta^4_-\otimes \Delta^6_+\oplus \Delta^4_-\otimes \Delta^6_-,$$
where $\Delta^n_\pm$ are copies of the half-spin representations of $\Spin_n$. 

The intersection of the spinor variety $\SS_{10}$ with $\PP(\delta_+^8)=\PP(\Delta^4_+\otimes \Delta^6_+)$ 
is then isomorphic to $\PP^1\times\PP^3$, since the spinor varieties of $\Spin_4=\SL_2\times\SL_2$ and $\Spin_6=\SL_4$
are copies of $\PP^1$ and $\PP^3$, respectively. In particular each $x\in\PP^1$ defines a two-dimensional isotropic subspace 
$U'(x)\subset U$,  and each $y\in\PP^3$ defines a three-dimensional isotropic subspace 
$U''(y)\subset U^\perp$. Hence an explicit isomorphism of $\PP^1\times\PP^3$ with $\SS_{10}\cap\PP(\delta_+^8)$
defined by sending $(x,y)$ to $U(x,y)=U'(x)\oplus U''(y)$. 

In particular we can recover $V'$ from  $\delta_+^8=\Delta^4_+\otimes \Delta^6_+$ through the formula
$$V'=\bigg\langle \bigcap_{y\in\PP^3}U(x,y)\bigg\rangle_{x\in\PP^1}.$$
Of course the same formula holds true if we replace $\delta_+^8$ by $\delta_-^8$. 

Thus the map $U\mapsto \{\delta_+^8, \delta_-^8\}$ of Proposition
\ref{8dec} is injective, and also the map $\{U,V,W\}\mapsto
\{\Delta_1,\Delta_2,\Delta_3,\Delta_4\}$ of Lemma \ref{4dec}, and
therefore $\varphi$ is injective as well. \qed

\medskip\noindent {\it Conclusion of the proof}.
At this point we know that there exists one suitable triple $\{A,B,C\}$, 
and possibly only finitely many others.
There remains to check that there is no other automorphism stabilizing
$P$ than the three involutions provided by the triple $\{A,B,C\}$. Recall that by \ref{concl:Spin_10},
any such automorphism must be an involution of the same type. 
So suppose that $R\subset\CC^{10}$ is another four-plane such that
$t_R$ stabilizes $P$.  Then $t_Ut_R$ must be of the same type (up to
scalar), in particular it must be an involution (up to scalar) and
there must exist $\kappa$ such that $t_Rt_U=\kappa t_Ut_R$. But then
$t_R$ induces an isomorphism between the eigenspaces of $t_U$ with
eigenvalues $\lambda$ and $\kappa\lambda$. Since $t_U$ has only two
eigenspaces and these have different dimensions, this implies that
$\kappa=1$, and $t_R$ commutes with $t_U$. Similarly, it commutes with
$t_V$ and $t_W$. 

Recall that the common eigenspaces decomposition of  $t_U, t_V, t_W$ is $A\oplus B\oplus C\oplus D$, where $D$ is the 
orthogonal to $A\oplus B\oplus C$. The involution $t_R$ can be diagonalized accordingly. Let us denote by $e_0,\ldots ,e_9$ a compatible basis, and by $\epsilon_0,\ldots ,\epsilon_9$ the corresponding eigenvalues of $t_R$. We can express the fact
that $t_R$ has eigenspaces of dimensions $4$ and $6$ by imposing that $tr(t_R)=2\theta$, with $\theta=\pm 1$. 
Similarly, we need $tr(t_Ut_R)=2\theta'$, with $\theta'=\pm 1$. But then $\epsilon_0+\cdots +\epsilon_3=\theta+\theta'$
and  $\epsilon_4+\cdots +\epsilon_9=\theta-\theta'$. So up to replacing $t_R$ by $-t_R$, and reordering, we must have 
$(\epsilon_0,\ldots ,\epsilon_3)=(1,1,1,-1)$ and $(\epsilon_4,\ldots ,\epsilon_9)=(1,1,1,-1,-1,-1)$, or 
$(\epsilon_0,\ldots ,\epsilon_3)=(1,1,-1,-1)$ and $(\epsilon_4,\ldots ,\epsilon_9)=(1,1,1,1,-1,-1)$. 
Performing the same analysis with $V$ and $W$, we conclude that there are (up to sign) only two possibilities for 
$(\epsilon_0,\ldots ,\epsilon_9)$, namely $(1,1,1,-1,1,-1,1,1,-1,-1)$ or $(1,-1,1,-1,1,-1,1,1,1,-1)$. 
Note that in the first case, we have split $B$ and $C$ into the sum of two orthogonal lines, and $D$ into the sum of two orthogonal planes; there are $1+1+4=6$ parameters for such splittings. In the second case, we split $A$, $B$ and $C$ into the sum of two orthogonal lines, and $D$ into the sum of a line and its  orthogonal hyperplane; there are $1+1+1+3=6$ parameters 
for such splittings. 

Then the choice of $P$ is more restricted:
the action of $t_R$ splits each $\Delta_i$ into two $2$-planes,
and $P$ is stabilized by $t_R$ as well if and only if it is the sum of
four lines chosen inside four of the resulting eight two-planes.
The number of parameters therefore drops from $12$ to $4$, 
and since $4+6<12$, we can conclude that $P$ cannot be generic if it
is stabilized by our extra $t_R$.  This ends the proof of
Theorem~\ref{t:g7.concl}.
\end{proof}

\begin{prop}
The fixed locus of any of the three involutions in the automorphism
group of a general six-dimensional Mukai variety
$X=\SS_{10}\cap P^{\perp}$, is the disjoint union of two surface rational
quartic scrolls. 
\end{prop}

\begin{proof}
Each automorphism is an involution 
$t=t_U$ associated to some four dimensional subspace $U\subset\CC^{10}$, on restriction to which the quadratic
form remains non-degenerate. Moreover the action of $t_U$ on $\Delta$ splits it into two eigenspaces $\Delta_+$ and $\Delta_-$, 
which as $\fso(U)\times \fso(U^{\perp})$ modules are tensor products of spin-modules. More concretely, recall that 
$\fso(U)\simeq\fso_4\simeq\fsl_2\times\fsl_2=\fsl(S)\times\fsl(T)$, with $U\simeq S\otimes T$ and $S$, $T$ the two spin two-dimensional modules. Moreover, another exceptional isomorphism yields $\fso(U^{\perp})\simeq\fso_6\simeq\fsl_4=\fsl(R)$,
with $U^{\perp}\simeq \wedge^2R$ and $R, R^\vee$ the two spin four-dimensional modules. In particular, 
$$\Delta_+\simeq R\otimes S, \qquad \mathrm{and} \qquad \Delta_-\simeq R^\vee\otimes T$$
(up to the exchange of $S$ and $T$). The fixed locus of the action of $t_U$ on $X$ is the union of its intersections
with $\PP(\Delta_+)$ and $\PP(\Delta_-)$. Since $P$ is the direct sum of the two-dimensional spaces $P_+\subset \Delta_+$ and
 $P_-\subset \Delta_-$, we get two disjoint subvarieties $\SS_\pm= \SS^{\pm}_{10}\cap \PP(P_\pm^{\perp})$, where 
 $\SS^{\pm}_{10}=\SS_{10}\cap \PP(\Delta_\pm)$. 
 
 There remains to identify these subvarieties. For that we just need to remember that $\SS_{10}\subset \PP(\Delta)$ is cut-out by quadrics, and that the quadrics vanishing on $\SS_{10}$ are parametrized by $\CC^{10}$ (see e.g.  \cite[5.1]{weyman}). In fact, given a vector 
 $v\in\CC^{10}$, the Clifford multiplication by $v$ sends $\Delta_+$ to $\Delta_-\simeq \Delta_+^\vee$ and we can let 
 $q_v(\delta)=\langle v.\delta,\delta\rangle $ for any spinor $\delta\in \Delta$. Now $S_+\subset\PP(\Delta_+)$ is  
 cut out by the restriction of those quadrics to $\Delta_+$. The decomposition $\CC^{10}=U\oplus U^{\perp}$ gives two 
 types of quadrics. For $v\in U\simeq S\otimes T$, the Clifford action of $v$ on $\Delta_+=R \otimes S$ maps it to 
 $R\otimes T\simeq R\otimes T^\vee$, and therefore the quadric $q_v$ vanishes identically on $\Delta_+$. But for 
 $v\in U^{\perp}\simeq\wedge^2R$, the Clifford action of $v$ on $\Delta_+$ sends it to its dual, and the quadric $q_v$
 is non zero. 
 So $S_+\subset\PP(\Delta_+)\simeq \PP(R\otimes S)$ is the locus cut-out by the quadrics parametrized by $\wedge^2R\simeq\wedge^2R^\vee$, and this locus is just $\PP(R)\times \PP(S)\simeq\PP^3\times\PP^1$. 
 Cutting it by $\PP(P_+)$, a 
 generic two-codimensional subspace, we get an irreducible surface which is a rational quartic scroll. 
\end{proof}

\medskip\noindent {\it Remark}. In genus $8$ and $9$ we have been able to understand part of the exceptional 
automorphism group as acting on some auxiliary elliptic curve, either by pointwise symmetries or translations 
by torsion points. In genus $7$ there is an associated abelian surface \cite[Theorem 9.5]{gsw}, but its 
construction is more involved 
and we have no geometric interpretation yet of the exceptional automorphism group as acting on this surface. 
%
%In a follow-up to this paper, we plan to investigate in more details the quite intriguing relationships between $\theta$-representations and Mukai varieties.

\bibliographystyle{alpha}

\medskip
\noindent
\textsc{Institut de Math{\'e}matiques de Toulouse ; UMR 5219,  UPS, F-31062 Toulouse Cedex 9, France}

\noindent
\textit{Email address}: \texttt{thomas.dedieu@math.univ-toulouse.fr, \\
  manivel@math.cnrs.fr}

\medskip

\bigskip

\newcommand{\NN}{\mathbb{N}}
\newcommand{\CCC}{\mathcal{C}}
\newcommand{\OOO}{\mathscr{O}}
\newcommand{\HHH}{\mathscr{H}}
\newcommand{\g}{\mathrm{g}}
\newcommand{\Eu}{\upchi_{\mathrm{top}}}
\newcommand{\J}{\operatorname{J}}
\newcommand{\Alb}{\operatorname{Alb}}
\newcommand{\Bs}{\operatorname{Bs}}
\newcommand{\Pic}{\operatorname{Pic}}
\newcommand{\F}{\mathrm{F}}

\bigskip 

\appendix
\section{Automorphism groups of prime Fano threefolds of genus twelve}
\smallskip
\begin{center}by \textsc{Yuri Prokhorov}\end{center}
\medskip
%, \linebreak by Yuri Prokhorov }
%\title{Appendix: Automorphism groups of prime Fano threefolds}
%\author{Yuri Prokhorov}
%\maketitle

\setcounter{section}{1}
\begin{mtheorem}
\label{mthm}
The automorphism group of a  general (in the moduli sense) prime Fano threefold of genus $12$ is trivial.
\end{mtheorem}

\begin{proof}

For a prime Fano threefold $X$ we denote by~$\F_1(X)$ 
the Hilbert scheme of lines, i.e. curves in~$X$ with Hilbert polynomial $h_1(t) = t + 1$. 
It is known that $\F_1(X)$ is of pure dimension $1$ (see e.g. \cite{Kuznetsov-Prokhorov-Shramov}).

\iffalse
We denote by
\begin{equation*}
\CCC_d(X) \subset X \times \F_d(X),
\end{equation*}
the universal curve.
Similarly, if $x \in X$ is a point, we denote by $\F_d(X,x) \subset \F_d(X)$ the subscheme of curves that pass through~$x$;
this is the fiber of the natural map $\CCC_d(X) \to X$ over~$x$.
\fi

\begin{claim}
\label{claim:g=12}
For any  prime Fano threefold  $X=X_{22}\subset \mathbb P^{13}$
the natural homomorphism 
\[
\Psi:\Aut(X)\to \Aut(\F_1(X))
\]
is injective. 
\end{claim}

\begin{proof}
Assume that  $\Psi$ is not injective. 
Take a non-trivial element $\varphi\in \operatorname{Ker}(\Psi)$. Thus $\varphi$  acts trivially on $\F_1(X)$.
 Fix a line $l\subset X$.
Apply the double projection \cite[Theorem~4.3.3, Theorem~4.3.7]{IP99} from $l$. 
This is the birational map 
\[
 \theta: X \dashrightarrow Y \subset \PP^6
\]
given by the linear system $|-K_X-2l|$ of 
hyperplane sections which are singular along $l$. 
Here $Y=Y_5\subset \PP^6$ is a smooth quintic del Pezzo threefold and the 
$\theta$-exceptional divisor is contracted 
to a rational normal quintic curve  $\Gamma \subset Y\subset \PP^6$. 
The map $\theta$ induces a $\varphi$-action on $Y\subset \PP^6$ 
by a projective transformation and the curve $\Gamma$ is $\varphi$-invariant.
A general line $l'\subset X$ is mapped to a line $m'\subset Y$ meeting $\Gamma$ at one point. 
The set of lines in  $Y$ passing through any point $y\in Y$ is finite (see e.g. \cite[Corollary~5.1.5]{Kuznetsov-Prokhorov-Shramov}).
 Since $\dim \F_1(X)=1$, 
the automorphism $\varphi$ acts trivially on $\Gamma$.
Thus the fixed point locus $Y^\varphi$ contains the hyperplane section $S:=Y\cap \langle\Gamma\rangle$. Recall that $H^2(Y,\ZZ)\simeq \Pic(Y)\simeq \ZZ$ and $H^3(Y,\ZZ)=0$
(see e.g. \cite[\S~12.2]{IP99}). Hence the induced action of $\varphi$ on $H^q(Y,\CC)$ is trivial for any $q$.

Assume that $\varphi$ is an element of finite order. Then its 
fixed point locus $Y^\varphi$ is smooth. 
Hence $Y^\varphi$ contains no one dimensional components (because $\uprho(Y)=1$) and  $S$ is a
smooth del Pezzo surface.  In particular, $\Eu(Y^\varphi)\ge 7$. 
This contradicts the topological Lefschetz fixed point formula \cite[Prop.~5.3.11]{Dieck1979}:
\[
\Eu(Y^\varphi)=\sum_q (-1)^q
\operatorname{Tr} \big(\varphi^*|_{H^q(Y,\CC)}  \big) =\sum_q (-1)^qh^q(Y,\CC)
=\Eu(Y) =4.
\]

Therefore $\varphi$ is an element of infinite order.
Any line on $Y$ meets $S$  hence $\varphi^m$ acts trivially 
on $\F_1(Y)$ for some $m$ (in fact, $m\le 3$).
Recall that there are exactly three lines in $Y$ passing through a general point $y\in Y$. This implies that $\varphi^m$ acts trivially on $Y$, a contradiction.
\end{proof}

Now we use  Mukai's realization of $X=X_{22}\subset \mathbb P^{13}$
as $VSP(C,6)$ where $C$ is a plane quartic
\cite{Mukai-1989}. 
Take a general quartic $C\subset \mathbb P^2$ and let $X=VSP(C,6)$. 
Then the curve $\F_1(X)$ is also a smooth plane quartic $F_C$ which is \textit{covariant} of $C$ \cite[Theorem~6.1]{Schreyer2001}.
The curve $\F_1(X)=F_C$ has a natural $(3,3)$-correspondence  of intersecting lines which defines an even theta characteristic $\Theta$ on $F_C$. There is a map $C \longmapsto (F_C,\Theta)$  of the corresponding moduli spaces which is called \textit{Scorza map}. It is birational \cite[Theorem~7.8]{Dolgachev-Kanev}. In particular, this implies that the curve $F_C$ is general in the moduli space of plane quartics. Since the plane quartic $F_C$ is general, we have  $\Aut(F_C)=\{1\}$. 
Hence  $\Aut(X)=\{1\}$ for $X=VSP(C,6)$ by Claim~\ref{claim:g=12}.
\end{proof}

\begin{mremark}
Note that in contrast with the cases $g\le 10$ the automorphism group 
of a prime Fano threefold of genus $g=12$ can be infinite.
We refer to \cite{Prokhorov-1990c}, \cite{Kuznetsov-Prokhorov-Shramov},% \cite{Matsumura1963/1964} 
\cite{Kuznetsov-Prokhorov} for description of infinite groups of automorphisms.
\end{mremark}

%\bibliography{/home/yuri/Dropbox/Documents/bib1/all,/home/yuri/Dropbox/Documents/bib1/prokho}
\bibliographystyle{alpha}

\medskip

\sc{Steklov Mathematical Institute of Russian Academy of Sciences $\&$ Laboratory of Algebraic Geometry, NRU HSE 
$\&$ Department of Algebra, Moscow State University, Moscow, Russia}

\it{Email address:} \tt{prokhoro@mi-ras.ru}
\end{document}